\documentclass[11pt]{article}
\usepackage{amsfonts, amsmath, amssymb, latexsym, eucal, amscd} 
\usepackage{cite}
\usepackage[all]{xy}
\newtheorem{theorem}{Theorem}[section]
\newtheorem{lemma}[theorem]{Lemma}
\newtheorem{prop}[theorem]{Proposition}
\newtheorem{corollary}[theorem]{Corollary}
\newtheorem{exAux}[theorem]{Example}
\newenvironment{example}{\begin{exAux} \rm}{\end{exAux}}
\newtheorem{Def}[theorem]{Definition}
\newenvironment{defi}{\begin{Def} \rm}{\end{Def}}
\newtheorem{Note}[theorem]{Note}
\newenvironment{note}{\begin{Note} \rm}{\end{Note}}
\newtheorem{Problem}[theorem]{Problem}

\newtheorem{Rem}[theorem]{Remark}

\newtheorem{Not}[theorem]{Notation}

\newtheorem{Conj}[theorem]{Conjecture}

\newtheorem{Ass}[theorem]{Assumption}

\newenvironment{proof}{\medskip\noindent{\bf Proof.\ }}{\qed\medskip}

\newcommand{\qed}{\hfill\mbox{$\Box$\qquad\qquad}}

\newcommand{\F}{\mathbb{F}}

\newcommand{\Matd}{\text{\rm Mat}_{d+1}(\F)}

\newcommand{\vphi}{\varphi}

\newcommand{\A}{{\mathcal A}}
\renewcommand{\b}[1]{\langle #1 \rangle }
\newif\ifDRAFT

\begin{document}

\title{Idempotent systems and  character algebras}

\author{Kazumasa Nomura and Paul Terwilliger}

\maketitle

\medskip

\begin{quote}
\small 
\begin{center}
\bf Abstract
\end{center}
We recently introduced the notion of an idempotent system.
This linear algebraic object is motivated by the structure of an association scheme.
There is a type of idempotent system, said to be symmetric.
In the present paper we classify up to isomorphism the idempotent systems and 
the symmetric idempotent systems.
We also describe how symmetric idempotent systems are related to character algebras.
\end{quote}

\section{Introduction}
\label{sec:intro}

A symmetric association scheme (or SAS) is a combinatorial object that generalizes a
distance-regular graph and a generously transitive permutation group.
The concept of an SAS first arose in design theory \cite{BM, BN, BS, RC} and group theory \cite{Wi}.
A systematic study began with \cite{Del, Hig}.
A comprehensive treatment is given in \cite{BI, BCN}.
The combinatorial regularity of an SAS gives it a rich algebraic structure.
A (symmetric) character algebra \cite{Kawada, BI} 
and an idempotent system \cite{NT:ips} can be used to study the algebraic structure of an SAS,
without having to assume the combinatorial structure.
A character algebra  is an abstraction of the Bose-Mesner algebra \cite{BN} of 
an SAS.
For a vertex $x$ of an SAS
the corresponding subconstituent algebra is generated by the
Bose-Mesner  algebra and the dual Bose-Mesner algebra with respect to $x$ \cite{T:subconst1}.
A symmetric idempotent system is an abstraction of the primary module of a subconstituent
algebra of an SAS \cite[Section 1]{NT:ips}.
Our purpose in the present paper is two fold:
(i) we classify up to isomorphism  the idempotent systems and symmetric
idempotent systems;
(ii) we describe how symmetric idempotent systems are related to
character algebras.

We recall the definition of an idempotent system \cite{NT:ips}.
Fix a field $\F$ and an integer $d \geq 0$.
Let $V$ denote a vector space over $\F$ with dimension $d+1$.
Let $\text{\rm End}(V)$ denote the $\F$-algebra consisting of the
$\F$-linear maps $V \to V$.
Let $\A$ denote an $\F$-algebra isomorphic to $\text{\rm End}(V)$.
An idempotent system in $\A$ is a sequence
$\Phi = ( \{E_i\}_{i=0}^d; \{E^*_i\}_{i=0}^d)$ such that
\begin{itemize}
\item[\rm (i)]
$\{E_i\}_{i=0}^d$ is a system of mutually orthogonal rank $1$ idempotents in $\A$;
\item[\rm (ii)]
$\{E^*_i\}_{i=0}^d$ is a system of mutually orthogonal rank $1$ idempotents in $\A$;
\item[\rm (iii)]
$E_0 E^*_i E_0 \neq 0 \quad (0 \leq i \leq d)$;
\item[\rm (iv)]
$E^*_0 E_i E^*_0 \neq 0 \quad (0 \leq i \leq d)$.
\end{itemize}
We say that $\Phi$ is over $\F$ and has diameter $d$. 
The idempotent system $\Phi$ is said to be symmetric whenever there exists 
an antiautomorphism of $\A$ that fixes each of $E_i$, $E^*_i$ for $0 \leq i \leq d$.

We recall the concept of  a character algebra \cite{BI, Kawada}.
Throughout this paper a character algebra is understood to be symmetric.
A character algebra over $\F$ is a sequence $({\mathcal C}; \{x_i\}_{i=0}^d)$
such that $\mathcal C$ is a commutative $\F$-algebra
and $\{x_i\}_{i=0}^d$ is a distinguished basis for $\mathcal C$
that satisfies the conditions in Definition \ref{def:SCalgebra} below.
We are mainly interested in the semisimple case, in which
$\mathcal C$ has a basis $\{e_i\}_{i=0}^d$ of primitive idempotents.
Among $\{e_i\}_{i=0}^d$ there exists a distinguished one, said to be trivial.
A character system over $\F$ is a sequence
$\Psi = ({\mathcal C} ; \{x_i\}_{i=0}^d; \{e_i\}_{i=0}^d)$,
where $({\mathcal C}; \{x_i\}_{i=0}^d)$ is a semisimple character algebra over $\F$,
and $\{e_i\}_{i=0}^d$ is an ordering of  the primitive idempotents of $\mathcal C$ with $e_0$ trivial.
We say that $\Psi$ has diameter $d$.

The following notation is convenient.
Let $\Matd$ denote the algebra consisting 
of the $d+1$ by $d+1$ matrices that have all entries in $\F$.
We index rows and columns by $0,1, \ldots, d$.
For $0 \leq i,j \leq d$ let $\Delta_{i,j}$ denote the matrix in $\Matd$
that has $(i,j)$-entry $1$ and all other entries $0$.

Our classification of idempotent systems is summarized as follows.
An invertible  matrix $R \in \Matd$ is said to be solid whenever 
for both $R$ and $R^{-1}$ all entries are nonzero in both column $0$ and row $0$.
Two matrices $R$, $S$ in $\Matd$ are said to be diagonally equivalent
whenever there exist invertible diagonal matrices $H$, $K$ in $\Matd$ such that
$S = H R K$.
For a solid invertible matrix $R \in \Matd$, we show that the sequence
\[
   \Phi_R = ( \{\Delta_{i,i} \}_{i=0}^d;  \{ R \Delta_{i,i} R^{-1} \}_{i=0}^d)
\]
is an idempotent system in $\Matd$.
We show that the map $R \mapsto \Phi_R$ induces
a bijection between the following two sets:
\begin{itemize}
\item[\rm (i)]
the diagonal equivalence classes of solid invertible matrices in $\Matd$;
\item[\rm (ii)]
the isomorphism classes of idempotent systems over $\F$ with diameter $d$.
\end{itemize}

Our classification of symmetric idempotent systems is summarized as follows.
An invertible $R \in \Matd$ is said to be almost orthogonal (AO) whenever
$R^{\sf t}$ is diagonally equivalent to $R^{-1}$.
By restricting the above bijection to AO solid invertible matrices,
we get a bijection between the following two sets:
\begin{itemize}
\item[\rm (i)]
the diagonal equivalence classes of AO solid invertible matrices in $\Matd$;
\item[\rm (ii)]
the isomorphism classes of symmetric idempotent systems over $\F$ with diameter $d$.
\end{itemize}

The above classifications have the following alternate version.
A solid invertible matrix $R \in \Matd$ is said to be normalized whenever
in column $0$ of $R$ all entries are equal to $1$ and in column $0$ of $R^{-1}$
all entries are the same.
As we will show, each diagonal equivalence class of solid invertible matrices contains 
a unique normalized element.
Thus our first bijection above induces a bijection between the following  two sets:
\begin{itemize}
\item[\rm (i)]
the normalized solid invertible matrices in $\Matd$;
\item[\rm (ii)]
the isomorphism classes of idempotent systems over $\F$ with diameter $d$.
\end{itemize}
Similarly we get  a bijection between the following two sets:
\begin{itemize}
\item[\rm (i)]
the AO normalized solid invertible matrices in $\Matd$;
\item[\rm (ii)]
the isomorphism classes of symmetric idempotent systems over $\F$ with diameter $d$.
\end{itemize}
Shortly we will describe this bijection in more detail.
Next we describe how AO normalized solid invertible matrices and symmetric idempotent
systems are related to character systems.
Consider the following sets:
\begin{center}
\begin{tabular}{rl}
$\text{\rm AON}_d(\F)$: &   \rule{0mm}{1ex}
the AO normalized solid invertible matrices in $\Matd$;
\\
$\text{\rm SIS}_d (\F)$: &    \rule{0mm}{3ex}
\parbox[t]{4.5in}{the isomorphism classes of symmetric idempotent systems over $\F$
with diameter $d$;}
\\
$\text{\rm CS}_d (\F)$: &  \rule{0mm}{3ex}
the isomorphism classes of character systems over $\F$ with diameter $d$.
\end{tabular}
\end{center}
We will show that these three sets are mutually in bijection,
and we will describe  the bijections involved.

We already gave a bijection from $\text{\rm AON}_d (\F)$ to $\text{\rm SIS}_d (\F)$.
We now describe the inverse bijection.
Let $\Phi = (\{E_i\}_{i=0}^d$; $\{E^*_i\}_{i=0}^d)$ denote a symmetric idempotent system in $\A$.
The elements $\{E_i\}_{i=0}^d$ form a basis for a commutative subalgebra $\mathcal M$ of $\A$.
For $0 \leq i \leq d$ there exists a unique $A_i \in {\mathcal M}$ such that
$A_i E^*_0 E_0 = E^*_i E_0$ for $0 \leq i \leq d$.
The elements $\{A_i\}_{i=0}^d$ form a basis of $\mathcal M$.
Let $P_\Phi$ denote the transition matrix
from the basis $\{E_i\}_{i=0}^d$ of $\mathcal M$ to the basis $\{A_i\}_{i=0}^d$ of $\mathcal M$.
The matrix $P_\Phi$ is AO normalized solid invertible,
and called the first eigenmatrix of $\Phi$.
For an idempotent system $\Phi$, let $[\Phi]$ denote the isomorphism class
of idempotent systems over $\F$ that contains $\Phi$.
We show that the following maps are inverses:
\[
\xymatrix@C=50pt{
 \text{\rm AON}_d(\F) 
 \ar[r]_{ P \; \mapsto \;[\Phi_P] }
&
  \text{\rm SIS}_d (\F),
}
\qquad\qquad
\xymatrix@C=50pt{
  \text{\rm SIS}_d (\F)
 \ar[r]_{ [\Phi] \; \mapsto \; P_\Phi}
&
  \text{\rm AON}_d (\F).
}
\]

Next we describe the bijective correspondence between $\text{\rm AON}_d (\F)$ 
and $\text{\rm CS}_d (\F)$.
Let $\Psi = ({\mathcal C};$ $\{x_i\}_{i=0}^d; \{e_i\}_{i=0}^d)$ denote a character system
over $\F$.
Let $P_\Psi$ denote the transition matrix from the basis $\{e_i\}_{i=0}^d$
of $\mathcal C$ to the basis $\{x_i\}_{i=0}^d$ of $\mathcal C$.
The matrix $P_\Psi$ is AO normalized solid invertible, and called the
eigenmatrix of $\Psi$.
For an AO normalized solid invertible $P \in \Matd$,
we construct a character system $\Psi_P$ as follows.
Let $\{e_i\}_{i=0}^d$ denote indeterminates, and let $\mathcal C$
denote the vector space over $\F$ with basis $\{e_i\}_{i=0}^d$.
Turn $\mathcal C$ into an algebra such that
$e_i e_j = \delta_{i,j} e_i$ for $0 \leq i \leq d$.
View $P$ as the transition matrix from $\{e_i\}_{i=0}^d$ to a basis
$\{x_i\}_{i=0}^d$.
We show that the sequence 
\[
  \Psi_P = ( {\mathcal C}; \{x_i\}_{i=0}^d; \{e_i\}_{i=0}^d)
\]
is a character system over $\F$.
We show that the following maps are inverses:
\[
\quad
\xymatrix@C=50pt{
 \text{\rm AON}_d (\F)
 \ar[r]_{  P \; \mapsto \; [\Psi_P] }
&
 \text{\rm CS}_d (\F),
}
\qquad\qquad
\xymatrix@C=50pt{
 \text{\rm CS}_d (\F)
 \ar[r]_{ [\Psi] \; \mapsto \; P_\Psi }
&
 \text{\rm AON}_d (\F).
}
\]

Next we describe the bijective correspondence between
$\text{\rm SIS}_d (\F)$ and $\text{\rm CS}_d (\F)$.
For a symmetric idempotent system $\Phi = (\{E_i\}_{i=0}^d; \{E^*_i\}_{i=0}^d)$ over $\F$,
we show that the sequence
\[
   \Psi_\Phi = ({\mathcal M}; \{A_i\}_{i=0}^d; \{E_i\}_{i=0}^d)
\]
is a character system over $\F$.
For a character system $\Psi = ( {\mathcal C}; \{x_i\}_{i=0}^d; \{e_i\}_{i=0}^d)$ over $\F$,
we construct a symmetric idempotent system $\Phi_\Psi$ as follows.
For $0 \leq i \leq d$ define $E_i$, $E^*_i \in \text{\rm End}({\mathcal C})$ such that
$E_i e_j = \delta_{i,j} e_j$ and $E^*_i x_j = \delta_{i,j} x_j$ for $0 \leq j \leq d$.
We show that the sequence
\[
  \Phi_\Psi = (\{E_i\}_{i=0}^d; \{E^*_i\}_{i=0}^d)
\]
is a symmetric idempotent system in $\text{\rm End}({\mathcal C})$.
We show that the following maps are inverses:
\[
\xymatrix@C=50pt{
 \text{\rm SIS}_d (\F)
 \ar[r]_{[\Phi]  \; \mapsto \: [\Psi_\Phi] }
&
 \text{\rm CS}_d (\F),
}
\qquad\qquad
\xymatrix@C=50pt{
 \text{\rm CS}_d (\F)
 \ar[r]_{ [\Psi] \; \mapsto \; [\Phi_\Psi]}
&
 \text{\rm SIS}_d (\F).
}
\]

For a symmetric idempotent system $\Phi = ( \{E_i\}_{i=0}^d; \{E^*_i\}_{i=0}^d)$,
the sequence 
\[
   \Phi^* = (\{E^*_i\}_{i=0}^d; \{E_i\}_{i=0}^d)
\]
is an idempotent system, called the dual of $\Phi$.
We have a bijective involution $\text{\rm SIS}_d (\F) \to \text{\rm SIS}_d (\F)$
that sends $[\Phi] \mapsto [\Phi^*]$, called the duality map.
From the above bijections $\text{\rm SIS}_d(\F) \to \text{\rm AON}_d (\F)$
and $\text{\rm SIS}_d(\F) \to \text{\rm CS}_d(\F)$,
the sets $\text{\rm AON}_d (\F)$ and $\text{\rm CS}_d (\F)$ inherit a duality map.
We describe these duality maps in detail.
We find that the duality map on $\text{\rm CS}_d (\F)$ is essentially the same thing as
the duality map for character algebras defined by Kawada \cite{Kawada}.

The paper is organized as follows.
In Section \ref{sec:pre} we fix some notation and recall some basic concepts.
In Section \ref{sec:ips} we recall the notion of an idempotent system.
In Section \ref{sec:solid} we classify the idempotent systems in terms of solid
invertible matrices.
In Section \ref{sec:symmetric} we classify the symmetric idempotent systems
in terms of AO solid invertible matrices.
In Section \ref{sec:normalized} we consider the normalization of a solid invertible
matrix, and we establish the bijection $\text{\rm AON}_d (\F) \to \text{\rm SIS}_d(\F)$.
In Section \ref{sec:known} we describe the inverse of this bijection.
In Sections \ref{sec:SCalgebra}--\ref{sec:eigen} we discuss character algebras
and character systems.
In Section \ref{sec:SISCS} we show that
$\text{\rm AON}_d (\F)$, $\text{\rm SIS}_d (\F)$, $\text{\rm CS}_d (\F)$
are mutually in bijection,
and we describe the bijections involved.
Section \ref{sec:dual} is about the duality maps for
$\text{\rm AON}_d (\F)$, $\text{\rm SIS}_d (\F)$, $\text{\rm CS}_d (\F)$.

\section{Preliminaries}
\label{sec:pre}

We now begin our formal argument.
In this section we fix some notation and recall some basic concepts.
Throughout this paper $\F$ denotes a field.
All vector spaces discussed in this paper are over $\F$.
All algebras discussed in this paper are associative,
over $\F$, and have a multiplicative identity.
For an algebra $\mathcal A$, 
by an {\em automorphism} of $\mathcal A$ we mean an algebra isomorphism ${\mathcal A} \to {\mathcal A}$,
and by an {\em antiautomorphism} of $\mathcal A$ we mean an $\F$-linear bijection
$\tau : {\mathcal A} \to {\mathcal A}$ such that $(Y Z)^\tau = Z^\tau Y^\tau$ for $Y, Z \in {\mathcal A}$.
For the rest of this paper, fix an integer $d \geq 0$.
Let $\Matd$ denote the $\F$-algebra consisting 
of the $d+1$ by $d+1$ matrices that have all entries in $\F$.
We index rows and columns by $0,1, \ldots, d$.
By the Skolem-Noether theorem \cite[Corollary 7.125]{Rot},
a map $\sigma :  \Matd \to \Matd$ is an automorphism of $\Matd$ if and only if
there exists an invertible $S \in \Matd$ such that $A^\sigma = S A S^{-1}$
for $A \in \Matd$.
A map $\tau : \Matd \to \Matd$ is an antiautomorphism of 
$\Matd$ if and only if there exists an invertible $T \in \Matd$ such that
$A^\tau = T A^{\sf t} T^{-1}$ for $A \in \Matd$, 
where $A^{\sf t}$ denotes the transpose of $A$.
For $0 \leq i,j \leq d$ let $\Delta_{i,j}$ denote the matrix in $\Matd$
that has $(i,j)$-entry $1$ and all other entries is $0$.
The matrices $\Delta_{i,j}$ $(0 \leq i,j \leq d)$ form a basis for the vector space $\Matd$.
Let $\mathcal D$ denote the subalgebra of $\Matd$ consisting of diagonal matrices in $\Matd$.
The algebra $\mathcal D$ is commutative,
and the matrices $\Delta_{i,i}$ $(0 \leq i \leq d)$ form a basis for the vector space $\mathcal D$.
Note that the identity matrix $I = \sum_{i=0}^d \Delta_{i,i}$.

\begin{lemma}    \label{lem:diag}    \samepage
\ifDRAFT {\rm lem:diag}. \fi
A matrix in $\Matd$ is diagonal if and only if it commutes with every diagonal matrix 
in $\Matd$.
\end{lemma}

Let $V$ denote a vector space with dimension $d+1$.
Let $\text{\rm End}(V)$ denote the algebra consisting of the $\F$-linear maps
$V \to V$.
We recall how each basis $\{v_i\}_{i=0}^d$ of $V$ gives an algebra isomorphism
$\text{\rm End}(V) \to \Matd$.
For $X \in \text{\rm End}(V)$ and $M \in \Matd$,
we say that {\em $M$ represents $X$ with respect to $\{v_i\}_{i=0}^d$}
whenever $X v_j = \sum_{i=0}^d M_{i,j} v_i$ for $0 \leq j \leq d$.
The isomorphism sends $X$ to the unique matrix in $\Matd$ that represents
$X$ with respect to $\{ v_i \}_{i=0}^d$.
For two bases $\{u_i\}_{i=0}^d$, $\{v_i\}_{i=0}^d$ of $V$,
by the {\em transition matrix from $\{u_i\}_{i=0}^d$ to $\{v_i\}_{i=0}^d$} we mean
the matrix $P \in \Matd$ such that
$v_j = \sum_{i=0}^d P_{i,j} u_i$ for $0 \leq j \leq d$.
In this case, $P$ is invertible, and $P^{-1}$ is the transition matrix
from $\{v_i\}_{i=0}^d$ to $\{u_i\}_{i=0}^d$.
We recall some basic facts concerning bilinear forms.
By a {\em bilinear form on $V$} we mean a map $\b{\; , \; } : V \times V \to \F$
that satisfies the following four conditions for $u,v,w \in V$ and $\alpha \in \F$:
(i) $\b{u+v,w} = \b{u,w} + \b{v,w}$;
(ii) $\b{\alpha u, v} = \alpha \b{u,v}$;
(iii) $\b{u,v+w} = \b{u,v} + \b{u,w}$;
(iv) $\b{u,\alpha v} = \alpha \b{u,v}$.
A bilinear form $\b{\; , \;}$ on $V$ is said to be {\em symmetric} whenever
$\b{u,v} = \b{v,u}$ for $u,v\in V$.
Let $\b{ \; , \; }$ denote a bilinear form on $V$.
The following are equivalent:
(i) there exists a nonzero $u \in V$ such that
$\b{ u, v} = 0$ for all $v \in V$;
(ii) there exists a nonzero $v \in V$ such that
$\b{ u,v } = 0$ for all $u \in V$.
The form $\b{ \; , \; }$ is said to be {\em degenerate}
whenever (i), (ii) hold and {\em nondegenerate} otherwise.
Assume that $\b{ \; , \;}$ is nondegenerate.
A basis $\{u_i\}_{i=0}^d$ of $V$ is said to be {\em orthogonal with respect to $\b{ \; , \;}$}
whenever $\b{u_i, u_j} = 0$ if $i \neq j$ $(0 \leq i,j \leq d)$.
In this case $\b{u_i, u_i} \neq 0$ for $0 \leq i \leq d$.

For the rest of this paper, let $\A$ denote an algebra that is isomorphic to $\Matd$.
The identity element of $\A$ is denoted by $I$.

\begin{defi}   \label{def:orth}    \samepage
\ifDRAFT {\rm def:orth}. \fi
By a {\em system of mutually orthogonal rank $1$ idempotents} in $\A$
we mean a sequence $\{E_i\}_{i=0}^d$ of elements in $\A$ such that
\[
   E_i E_j = \delta_{i,j} E_i  \qquad\qquad   (0 \leq i,j \leq d),
\]
\[
  \text{\rm rank} (E_i) = 1 \qquad\qquad (0 \leq i \leq d).
\]
\end{defi}

\begin{example}    \label{example}    \samepage
\ifDRAFT {\rm example}. \fi
The matrices $\{\Delta_{i,i} \}_{i=0}^d$ form a system of mutually orthogonal
rank $1$ idempotents in $\Matd$.
\end{example}

We now make a more general statement.

\begin{lemma}    \label{lem:R}     \samepage
\ifDRAFT {\rm lem:R}. \fi
For a sequence $\{E_i\}_{i=0}^d$ of elements in $\Matd$
the following are equivalent:
\begin{itemize}
\item[\rm (i)]
$\{E_i\}_{i=0}^d$ is a system of mutually orthogonal rank $1$ idempotents;
\item[\rm (ii)]
there exists an invertible $R \in \Matd$ such that
$E_i = R \Delta_{i,i} R^{-1}$ for $0 \leq i \leq d$.
\end{itemize}
\end{lemma}

The above result can be stated more abstractly as follows.

\begin{lemma}    \label{lem:E}    \samepage
\ifDRAFT {\rm lem:E}. \fi
For a sequence $\{E_i\}_{i=0}^d$ of elements in $\A$ the following are equivalent:
\begin{itemize}
\item[\rm (i)]
$\{E_i\}_{i=0}^d$ is a system of mutually orthogonal rank $1$ idempotents;
\item[\rm (ii)]
there exists an algebra isomorphism $\A \to \Matd$ that sends
$E_i \mapsto \Delta_{i,i}$ for $0 \leq i \leq d$.
\end{itemize}
\end{lemma}

\begin{lemma}    \label{lem:Ei}    \samepage
\ifDRAFT {\rm lem:Ei}. \fi
Let $\{E_i\}_{i=0}^d$ denote a system of mutually orthogonal rank $1$
idempotents in $\A$.
Then $\{E_i\}_{i=0}^d$ form a basis for a commutative subalgebra of $\A$.
Moreover $I = \sum_{i=0}^d E_i$.
\end{lemma}

In this paper we will occasionally speak of three sets being
mutually in bijection.
This means that for any ordering $A, B, C$ of the three sets
and $a \in A$, $b \in B$, $c \in C$,
if $a$, $b$ correspond and $b$, $c$ correspond then
$a$, $c$ correspond.

\section{Idempotent systems}
\label{sec:ips}

In this section we recall from \cite{NT:ips} the notion of an idempotent system,
and make some general remarks about it.
Recall the algebra $\A$ that is isomorphic to $\Matd$.

\begin{defi}    {\rm (See \cite[Definition 3.1]{NT:ips}.) }
\label{def:ips}    \samepage
\ifDRAFT {\rm def:ips}. \fi
By an {\em idempotent system in $\A$}  we mean a sequence
\[
  \Phi =  (\{E_i\}_{i=0}^d; \{E^*_i\}_{i=0}^d)
\]
such that
\begin{itemize}
\item[\rm (i)]
$\{E_i\}_{i=0}^d$ is a system of mutually orthogonal rank $1$ idempotents in $\A$;
\item[\rm (ii)]
$\{E^*_i\}_{i=0}^d$ is a system of mutually orthogonal rank $1$ idempotents in $\A$;
\item[\rm (iii)]
$E_0 E^*_i E_0 \neq 0 \quad (0 \leq i \leq d)$;
\item[\rm (iv)]
$E^*_0 E_i E^*_0 \neq 0 \quad (0 \leq i \leq d)$.
\end{itemize}
The idempotent system $\Phi$ is said to be {\em over $\F$}.
We call $d$ the {\em diameter} of $\Phi$.
\end{defi}

Let $\Phi = (\{E_i\}_{i=0}^d; \{E^*_i\}_{i=0}^d)$ denote an idempotent system in $\A$.
Then the sequence
\[
   \Phi^* = (\{E^*_i\}_{i=0}^d; \{E_i\}_{i=0}^d)
\]
is an idempotent system in $\A$, called the {\em dual of $\Phi$}.
For an object $\phi$ attached to $\Phi$, the corresponding object attached to $\Phi^*$
is denoted by $\phi^*$.

For an algebra $\A'$ and an algebra isomorphism $\sigma : \A \to \A'$,
define the sequence
\[
   \Phi^\sigma = (\{E_i^\sigma\}_{i=0}^d; \{(E^*_i)^\sigma \}_{i=0}^d).
\]
Then $\Phi^\sigma$ is an idempotent system in $\A'$.
Let $\Phi' = (\{E'_i\}_{i=0}^d; \{E^{* \prime}\}_{i=0}^d)$
denote an idempotent system in $\A'$.
By an {\em isomorphism of idempotent systems from $\Phi$ to $\Phi'$} we mean
an algebra isomorphism $\sigma : \A \to \A'$ such that $\Phi^\sigma = \Phi'$.
The idempotent systems $\Phi$ and $\Phi'$ are said to be {\em isomorphic} whenever there exists an
isomorphism of idempotent systems from $\Phi$ to $\Phi'$.

\begin{defi}    \label{def:M}    \samepage
\ifDRAFT {\rm def:M}. \fi
Let $\Phi = (\{E_i\}_{i=0}^d; \{E^*_i\}_{i=0}^d)$ denote an idempotent system in $\A$.
Recall from Lemma \ref{lem:Ei} that $\{E_i\}_{i=0}^d$ form a basis for a commutative
subalgebra of $\A$;
we denote this subalgebra by $\mathcal M$.
\end{defi}

\section{A classification of the idempotent systems}
\label{sec:solid}

In this section we classify the idempotent systems up to isomorphism,
in terms of a type of invertible matrix said to be solid.
To motivate the classification,
we first construct an example of an idempotent system.
Given an invertible $R \in \Matd$,
consider the following matrices in $\Matd$:
\begin{align}
  E_i &= \Delta_{i,i},    &   E^*_i &= R \Delta_{i,i} R^{-1}           &&   (0 \leq i \leq d).   \label{eq:defEiEsi}
\end{align}
Note that each of $\{E_i\}_{i=0}^d$ and $\{E^*_i\}_{i=0}^d$ is a system of mutually
orthogonal rank $1$ idempotents in $\Matd$.
Our next goal is to find a necessary and sufficient condition on $R$
such that the conditions Definition \ref{def:ips}(iii), (iv) are satisfied.

\begin{lemma}    \label{lem:PhiRpre}    \samepage
\ifDRAFT {\rm lem:PhiRpre}. \fi
Referring to \eqref{eq:defEiEsi},  
the entries of $E_0 E^*_i E_0$ and $E^*_0 E_i E^*_0$ are described as follows.
For $0 \leq r,s \leq d$ their $(r,s)$-entry is
\begin{align}
 (E_0 E^*_i E_0)_{r,s} &=  \delta_{r,0} \delta_{s,0} R_{0,i} (R^{-1})_{i,0},    \label{eq:canaux1}
\\
 (E^*_0 E_i E^*_0)_{r,s} &= R_{r,0} (R^{-1})_{0,i} R_{i,0} (R^{-1})_{0,s}.      \label{eq:canaux2}
\end{align}
\end{lemma}

\begin{proof}
By matrix multiplication.
\end{proof}

\begin{lemma}    \label{lem:PhiRpre2}    \samepage
\ifDRAFT {\rm lem:PhiRpre2}. \fi
Referring to \eqref{eq:defEiEsi},
for $0 \leq i \leq d$ the following are equivalent:
\begin{itemize}
\item[\rm (i)]
$E_0 E^*_i E_0 \neq 0$;
\item[\rm (ii)]
$R_{0,i} \neq 0$ and $(R^{-1})_{i,0} \neq 0$.
\end{itemize}
\end{lemma}

\begin{proof}
Use \eqref{eq:canaux1}.
\end{proof}

\begin{lemma}    \label{lem:PhiRpre3}    \samepage
\ifDRAFT {\rm lem:PhiRpre3}. \fi
Referring to \eqref{eq:defEiEsi}, for $0 \leq i \le d$ the following are equivalent:
\begin{itemize}
\item[\rm (i)]
$E^*_0 E_i E^*_0 \neq 0$;
\item[\rm (ii)]
$R_{i,0} \neq 0$ and $(R^{-1})_{0,i} \neq 0$.
\end{itemize}
\end{lemma}

\begin{proof}
Use \eqref{eq:canaux2}.
\end{proof}

\begin{defi}    \label{def:solid}     \samepage
\ifDRAFT {\rm def:solid}. \fi
An invertible matrix $R \in \Matd$ is said to be {\em solid} whenever
the following hold:
\begin{itemize}
\item[\rm (i)]
in column $0$ and row $0$ of $R$ all entries are nonzero;
\item[\rm (ii)]
in column $0$ and row $0$ of $R^{-1}$ all entries are nonzero.
\end{itemize}
\end{defi}

\begin{prop}    \label{prop:canonical2}    \samepage
\ifDRAFT {\rm prop:canonical2}. \fi
Referring to \eqref{eq:defEiEsi},
the following are equivalent:
\begin{itemize}
\item[\rm (i)]
the sequence $(\{E_i\}_{i=0}^d; \{E^*_i\}_{i=0}^d)$ is an idempotent system in $\Matd$;
\item[\rm (ii)]
the matrix $R$ is solid.
\end{itemize}
\end{prop}

\begin{proof}
By Definitions \ref{def:ips}, \ref{def:solid} and
Lemmas \ref{lem:PhiRpre2}, \ref{lem:PhiRpre3}.
\end{proof}

\begin{defi}   \label{def:PhiR}   \samepage
\ifDRAFT {\rm def:PhiR}. \fi
For a solid invertible matrix $R \in \Matd$ define the sequence
\[
  \Phi_R =  (\{\Delta_{i,i}\}_{i=0}^d;  \{R \Delta_{i,i} R^{-1}\}_{i=0}^d).
\]
Note by Proposition \ref{prop:canonical2} that $\Phi_R$ is an idempotent system in $\Matd$.
\end{defi}

\begin{prop}    \label{prop:canonical0}     \samepage
\ifDRAFT {\rm lem:canonical0}. \fi
Every idempotent system in $\A$ is isomorphic to $\Phi_R$ for some
solid invertible $R \in \Matd$.
\end{prop}

\begin{proof}
Let $\Phi = (\{E_i\}_{i=0}^d; \{E^*_i\}_{i=0}^d)$ denote an idempotent system in $\A$.
By Lemma \ref{lem:E}  there exists an algebra isomorphism $\sigma : \A \to \Matd$ that sends
$E_i \mapsto \Delta_{i,i}$ for $0 \leq i \leq d$.
Then $\Phi^\sigma$ is an idempotent system in $\Matd$, and $\sigma$ is an isomorphism
of idempotent systems from $\Phi$ to $\Phi^\sigma$.
By Lemma \ref{lem:R} there exists an invertible $R \in \Matd$
such that $(E^*_i)^\sigma = R \Delta_{i,i} R^{-1}$ for $0 \leq i \leq d$.
The idempotent system $\Phi^\sigma$ has the form
\[
   \Phi^\sigma = (\{\Delta_{i,i} \}_{i=0}^d;  \{ R \Delta_{i,i} R^{-1} \}_{i=0}^d).
\]
The matrix $R$ is solid by Proposition \ref{prop:canonical2},
and $\Phi^\sigma = \Phi_R$ by Definition \ref{def:PhiR}.
The result follows.
\end{proof}

\begin{defi}    \label{def:diagequiv}     \samepage
\ifDRAFT {\rm def:diagequiv}. \fi
Matrices $S$, $T$ in $\Matd$ are said to be {\em diagonally equivalent}
whenever there exist invertible diagonal matrices $H$, $K$ in $\Matd$ such that
$T = H S K$.
\end{defi}

Note that diagonal equivalence is an equivalence relation on $\Matd$.

Let $R$ denote a solid invertible matrix in $\Matd$ and let $H$, $K$ denote
invertible diagonal matrices in $\Matd$.
For $0 \leq r,s \leq d$ the $(r,s)$-entries of 
$H R K$ and $(H R K)^{-1}$ are 
\begin{align}
  (H R K)_{r,s} &= R_{r,s}  H_{r,r}  K_{s,s},   &
  \big( (H R K)^{-1}\big)_{r,s} &= \frac{ (R^{-1})_{r,s} } {H_{s,s}  K_{r,r} }.         \label{eq:HRK}
\end{align}

\begin{lemma}    \label{lem:edgeequiv}     \samepage
\ifDRAFT {\rm lem:edgeequiv}. \fi
Let $R$ denote a solid invertible matrix in $\Matd$.
Then every matrix that is diagonally equivalent to $R$ is solid invertible.
\end{lemma}

\begin{proof}
Use \eqref{eq:HRK}.
\end{proof}

\begin{prop}    \label{prop:RS}    \samepage
\ifDRAFT {\rm prop:RS}. \fi
For solid invertible matrices $R$, $S$ in $\Matd$
the following are equivalent:
\begin{itemize}
\item[\rm (i)]
$R$ and $S$ are diagonally equivalent;
\item[\rm (ii)]
the idempotent systems $\Phi_R$ and $\Phi_S$ are isomorphic.
\end{itemize}
Suppose {\rm (i)}, {\rm (ii)} hold.
Let $H$, $K$ denote invertible diagonal matrices in $\Matd$ such that $S = H R K$.
Then the automorphism of $\Matd$ that sends $A \mapsto H A H^{-1}$ is an isomorphism of
idempotent systems from $\Phi_R$ to $\Phi_S$.
\end{prop}

\begin{proof}
(i) $\Rightarrow$ (ii)
There exist invertible diagonal matrices $H$, $K$ in $\Matd$ such that $S = H R K$.
Consider the automorphism $\sigma$ of $\Matd$ that sends $A \mapsto H A H^{-1}$ for $A \in \Matd$.
We show that $\sigma$ is an isomorphism of idempotent systems from $\Phi_R$ to $\Phi_S$.
For $0 \leq i \leq d$ we have $H \Delta_{i,i} H^{-1} = \Delta_{i,i}$ by Lemma \ref{lem:diag},
and so $\Delta_{i,i}^\sigma = \Delta_{i,i}$.
Using $H R = S K^{-1}$ and Lemma \ref{lem:diag}, we find that for $0 \leq i \leq d$,
\[
 H R \Delta_{i,i} R^{-1} H^{-1}
  = S K^{-1} \Delta_{i,i} K S^{-1} = S \Delta_{i,i} S^{-1}.
\]
So $\sigma$ sends $R \Delta_{i,i} R^{-1} \mapsto S \Delta_{i,i} S^{-1}$.
By these comments $\sigma$ is an isomorphism of idempotent systems from $\Phi_R$ to $\Phi_S$.

(ii) $\Rightarrow$ (i)
Let $\sigma$ denote an isomorphism of idempotent systems from $\Phi_R$ to $\Phi_S$.
By the Skolem-Noether theorem, there exists an invertible $H \in \Matd$
such that $A^\sigma = H A H^{-1}$ for $A \in \Matd$.
For $0 \leq i \leq d$, $\sigma$ fixes $\Delta_{i,i}$ and so $H \Delta_{i,i} H^{-1} = \Delta_{i,i}$.
By this and Lemma \ref{lem:diag}, $H$ is diagonal.
For $0 \leq i \leq d$, $\sigma$ sends $R \Delta_{i,i}R^{-1} \mapsto S \Delta_{i,i} S^{-1}$,
so
\[
  H R \Delta_{i,i} R^{-1} H^{-1} = S \Delta_{i,i} S^{-1}.
\]
Thus  $R^{-1} H^{-1} S$ commutes with $\Delta_{i,i}$ for $0 \leq i \leq d$.
By this and Lemma \ref{lem:diag}, $R^{-1} H^{-1} S$ is diagonal;
denote this diagonal matrix by $K$.
Then $S = H R K$.
The matrices $H$, $K$ are diagonal, so $R$ and $S$ are diagonally equivalent.

Suppose (i), (ii) hold.
In the proof of (i) $\Rightarrow$ (ii), we have shown the last assertion of the
proposition statement.
\end{proof}

\begin{defi}    \label{def:Phirho}    \samepage
\ifDRAFT {\rm def:Phirho}. \fi
Let $\rho$ denote a diagonal equivalence class of solid invertible matrices in $\Matd$.
By Proposition \ref{prop:RS} the set $\{ \Phi_R \,|\, R \in \rho\}$ is contained in 
an isomorphism class of idempotent systems;
denote this isomorphism class by $\Phi_\rho$.
\end{defi}

In the next result we classify the idempotent systems up to isomorphism.

\begin{theorem}    \label{thm:main1}    \samepage
\ifDRAFT {\rm thm:main1}. \fi
Consider the following sets:
\begin{itemize}
\item[\rm (i)]
the diagonal equivalence classes of solid invertible matrices in $\Matd$;
\item[\rm (ii)]
the isomorphism classes of idempotent systems over $\F$ with diameter $d$.
\end{itemize}
The map $\rho \mapsto \Phi_\rho$ is a bijection from {\rm (i)} to {\rm (ii)}.
\end{theorem}

\begin{proof}
The given map is surjective by Proposition \ref{prop:canonical0}
and injective by Proposition \ref{prop:RS}.
\end{proof}

We have some comments.

\begin{lemma}    \label{lem:Rinv}    \samepage
\ifDRAFT {\rm lem:Rinv}. \fi
For a solid invertible $R \in \Matd$
the matrix $R^{-1}$ is solid invertible.
Moreover the map $\Matd \to \Matd$,  $A \mapsto R A R^{-1}$ is an isomorphism of
idempotent systems from $\Phi_{R^{-1}}$ to $(\Phi_R)^*$.
\end{lemma}

\begin{proof}
The matrix $R^{-1}$ is solid invertible by Definition \ref{def:solid}.
We have
\begin{align*}
 \Phi_{R^{-1}} &= (\{\Delta_{i,i}\}_{i=0}^d; \{ R^{-1} \Delta_{i,i} R \}_{i=0}^d),   &
 (\Phi_R)^* &= (\{R \Delta_{i,i} R^{-1} \}_{i=0}^d; \{\Delta_{i,i}\}_{i=0}^d).
\end{align*}
Thus the map $A \mapsto R A R^{-1}$ is an isomorphism of idempotent systems
from $\Phi_{R^{-1}}$ to $(\Phi_R)^*$.
\end{proof}

For a solid invertible $R \in \Matd$, note by Definition \ref{def:solid}
that $R^{\sf t}$ is solid invertible.
In Section \ref{sec:symmetric} we consider the case in which $R^{\sf t}$ is diagonally
equivalent to $R^{-1}$.

\section{Symmetric idempotent systems}
\label{sec:symmetric}

In \cite{NT:ips} we introduced a type of idempotent system,
said to be symmetric.
In this section we classify the symmetric idempotent systems
in terms of solid invertible matrices.

\begin{defi}  {\rm (See \cite[Definition 5.1]{NT:ips}.)  } 
\label{def:sym}    \samepage
\ifDRAFT {\rm def:sym}. \fi
Let $\Phi = (\{E_i\}_{i=0}^d; \{E^*_i\}_{i=0}^d)$ denote an idempotent system in $\A$.
We say that $\Phi$ is {\em symmetric} whenever there exists an antiautomorphism $\dagger$ of $\A$
that fixes each of $E_i$, $E^*_i$ for $0 \leq i \leq d$.
\end{defi}

\begin{lemma}    {\rm (See \cite[Lemma 5.2]{NT:ips}.) }
Referring to Definition \ref{def:sym},
the antiautomorphism $\dagger$ is unique and $(A^\dagger)^\dagger  = A$ for $A \in \A$.
\end{lemma} 

\begin{prop}    \label{prop:canonical3}    \samepage
\ifDRAFT {\rm prop:canonical3}. \fi
For a solid invertible $R \in \Matd$ the following are equivalent:
\begin{itemize}
\item[\rm (i)]
$R^{\sf t}$ is diagonally equivalent to $R^{-1}$;
\item[\rm (ii)]
the idempotent system $\Phi_R$ is symmetric.
\end{itemize}
Suppose {\rm (i)}, {\rm (ii)} hold.
Let $H$, $K$ denote invertible diagonal matrices in $\Matd$ such that $R^{\sf t} = H R^{-1} K$.
Let $\dagger$ denote the antiautomorphism of $\Matd$ corresponds to $\Phi_R$.
Then $\dagger$ sends $A \mapsto K A^{\sf t} K^{-1}$ for $A \in \Matd$.
\end{prop}

\begin{proof}
For $0 \leq i \leq d$ define
\begin{align*}
  E_i &= \Delta_{i,i},   &   E^*_i &= R \Delta_{i,i} R^{-1}.
\end{align*}
Note that $\Phi_R = (\{E_i\}_{i=0}^d; \{E^*_i\}_{i=0}^d)$.

(i) $\Rightarrow$ (ii)
Let $H$, $K$ denote invertible diagonal matrices in $\Matd$ such that $R^{\sf t} = H R^{-1} K$.
Let $\dagger$ denote the antiautomorphism of $\Matd$ that sends $A \mapsto K A^{\sf t} K^{-1}$.
We show that $\dagger$ fixes each of $E_i$ and $E^*_i$ for $0 \leq i \leq d$.
Using Lemma \ref{lem:diag} we find that $\dagger$ fixes $E_i$ for $0 \leq i \leq d$.
Observe that $\dagger$ sends $E^*_i$ to
$K (R^{-1})^{\sf t} \Delta_{i,i} R^{\sf t} K^{-1}$ for $0 \leq i \leq d$.
Using $R^{\sf t} K^{-1} = H R^{-1}$ and Lemma \ref{lem:diag}, 
\begin{align*}
 K (R^{-1})^{\sf t} \Delta_{i,i} R^{\sf t} K^{-1} &= R \Delta_{i,i} R^{-1}  && (0 \leq i \leq d).
\end{align*}
By these comments $\dagger$ fixes $E^*_i$ for $0 \leq i \leq d$.
We have shown that $\dagger$ fixes each of $E_i$ and $E^*_i$ for $0 \leq i \leq d$.
Therefore $\Phi_R$ is symmetric.

(ii) $\Rightarrow$ (i)
By Definition \ref{def:sym} there exists an antiautomorphism $\dagger$ of $\Matd$
that fixes each of $E_i$ and $E^*_i$ for $0 \leq i \leq d$.
By our comments in Section \ref{sec:pre} there exists an invertible $K \in \Matd$
such that $A^\dagger = K A^{\sf t} K^{-1}$ for $A \in \Matd$.
The matrix $K$ is diagonal by Lemma \ref{lem:diag} and since
$\dagger$ fixes $E_i$ for $0 \leq i \leq d$.
Observe that $\dagger$ sends $E^*_i$ to 
$K (R^{-1})^{\sf t} \Delta_{i,i} R^{\sf t} K^{-1}$ for $0 \leq i \leq d$.
By this and since $\dagger$ fixes $E^*_i$,
\begin{align*}
  R \Delta_{i,i} R^{-1} &= K (R^{-1})^{\sf t} \Delta_{i,i} R^{\sf t} K^{-1}    && (0 \leq i \leq d).
\end{align*}
Thus $R^{\sf t} K^{-1} R$ commutes with $\Delta_{i,i}$ for $0 \leq i \leq d$.
By this and Lemma \ref{lem:diag}, $R^{\sf t} K^{-1} R$ is diagonal;
denote this diagonal matrix by $H$.
Then $R^{\sf t} = H R^{-1} K$.
So $R^{\sf t}$ is diagonally equivalent to $R^{-1}$.

Suppose (i), (ii) hold.
In the proof of (i) $\Rightarrow$ (ii),
we have shown the last assertion of the proposition statement.
\end{proof}

In view of Proposition \ref{prop:canonical3}, we make a definition.

\begin{defi}    \label{def:feasible}    \samepage
\ifDRAFT {\rm def:feasible}. \fi
A matrix $R \in \Matd$ is said to be {\em almost orthogonal (AO)}
whenever $R$ is invertible and $R^{\sf t}$ is diagonally equivalent to $R^{-1}$.
\end{defi}

\begin{lemma}    \label{lem:genorth3}   \samepage
\ifDRAFT {\rm lem:genorth3}. \fi
Let $R$ denote an AO matrix in $\Matd$.
Then every matrix in $\Matd$ that is diagonally equivalent to $R$ is AO.
\end{lemma}

\begin{proof}
By Definition \ref{def:feasible} there exist invertible diagonal matrices $H_1$, $K_1$
in $\Matd$ such that $R^{\sf t} = H_1 R^{-1} K_1$.
Let $S$ denote a matrix in $\Matd$ that is diagonally equivalent to $R$.
Then there exist invertible diagonal matrices $H_2$, $K_2$ in $\Matd$ such that
$S = H_2 R K_2$.
One routinely finds that $S^{\sf t} = K_2^2 H_1 S^{-1} K_1 H_2^2$.
Therefore $S^{\sf t}$ is diagonally equivalent to $S^{-1}$.
\end{proof}

In the next result we classify up to isomorphism the symmetric idempotent systems.

\begin{theorem}    \label{thm:main5}    \samepage
\ifDRAFT {\rm thm:main5}. \fi
Consider the following sets:
\begin{itemize}
\item[\rm (i)]
the diagonal equivalence classes of AO solid invertible matrices in $\Matd$;
\item[\rm (ii)]
the isomorphism classes of  symmetric idempotent systems over $\F$ with diameter $d$.
\end{itemize}
The map $\rho \mapsto \Phi_\rho$ is a bijection from {\rm (i)} to {\rm (ii)}.
\end{theorem}

\begin{proof}
By Theorem \ref{thm:main1}, Proposition \ref{prop:canonical3}, and Lemma \ref{lem:genorth3}.
\end{proof}

\section{Normalized solid invertible matrices}
\label{sec:normalized}

In Section \ref{sec:solid},
we classified the idempotent systems up to isomorphism.
We showed that the isomorphism classes are in bijection with the diagonal equivalence
classes of solid invertible matrices.
In this section we introduce a type of solid invertible matrix,
said to be normalized.
We show that each diagonal equivalence class of solid invertible matrices
contains a unique normalized element.

\begin{defi}    \label{def:normalized}    \samepage
\ifDRAFT {\rm def:normalized}.  \fi
A solid invertible matrix $R \in \Matd$ is said to be {\em normalized}
whenever the following {\rm (i), (ii)} hold:
\begin{itemize}
\item[\rm (i)]
in column $0$ of $R$ all entries are equal to $1$;
\item[\rm (ii)]
in column $0$ of $R^{-1}$ all entries are the same.
\end{itemize}
\end{defi}

\begin{lemma}    \label{lem:normalized}    \samepage
\ifDRAFT {\rm lem:normalized}. \fi
Let $R$ denote a solid invertible matrix in $\Matd$ and let
$H$, $K$ denote invertible diagonal matrices in $\Matd$,
Then $H R K$ is normalized if and only if
\begin{align}
   H_{r,r} &= \frac{ 1 } { R_{r,0} K_{0,0} },    &
  K_{r,r} &= \frac{ (R^{-1})_{r,0} K_{0,0} } { (R^{-1})_{0,0} }  &&  (0 \leq r \leq d).     \label{eq:Kii}
\end{align}
\end{lemma}

\begin{proof}
Use \eqref{eq:HRK} with $s=0$, along with Definition \ref{def:normalized}.
\end{proof}

\begin{prop}    \label{prop:normalized}    \samepage
\ifDRAFT {\rm prop:normalized}. \fi
Each diagonal equivalence class of solid invertible matrices in $\Matd$
contains a unique normalized element.
\end{prop}

\begin{proof}
Let $R$ denote a solid invertible matrix in $\Matd$.
We show that there exists a unique normalized solid invertible matrix in $\Matd$ that is diagonally equivalent to $R$.
Observe that there exist invertible diagonal matrices $H$, $K$ in $\Matd$ that satisfy \eqref{eq:Kii}.
Then $H R K$ is a normalized solid invertible matrix in $\Matd$ that is diagonally equivalent to $R$.
Concerning uniqueness,
let $H'$, $K'$ denote invertible diagonal matrices in $\Matd$ such that $H' R K'$ is normalized.
By Lemma \ref{lem:normalized} there exists a nonzero $\alpha \in \F$ such that
$H' = \alpha H$ and $K' = \alpha^{-1} K$.
Therefore $H R K = H' R K'$.
\end{proof}

\begin{defi}     \label{def:[Phi]}     \samepage
\ifDRAFT {\rm def:[Phi]}. \fi
For an idempotent system $\Phi$ over $\F$ with diameter $d$,
let $[\Phi]$ denote the isomorphism class that contains $\Phi$.
\end{defi}

\begin{corollary}      \label{cor:main1}    \samepage
\ifDRAFT {\rm cor:main1}. \fi
Consider the following sets:
\begin{itemize}
\item[\rm (i)] the normalized solid invertible matrices in $\Matd$;
\item[\rm (ii)]
the isomorphism classes of idempotent systems over $\F$ with diameter $d$.
\end{itemize}
The map $R \mapsto [\Phi_R]$ is a bijection from {\rm (i)} to {\rm (ii)}.
\end{corollary}

\begin{proof}
By Theorem \ref{thm:main1} and Proposition \ref{prop:normalized}.
\end{proof}

\begin{defi}    \label{def:AONSIPS}    \samepage
\ifDRAFT {\rm def:AONSIPS}. \fi
Let $\text{\rm AON}_d(\F)$ denote the set consisting of
the AO normalized solid invertible matrices in $\Matd$.
Let $\text{\rm SIS}_d(\F)$ denote the set consisting of
the isomorphism classes of symmetric idempotent systems over $\F$ with 
diameter $d$.
\end{defi}

\begin{corollary}    \label{cor:main2} \samepage
\ifDRAFT {\rm cor:main2}. \fi
The map $\text{\rm AON}_d(\F) \to \text{\rm SIS}_d(\F)$, $R \mapsto [\Phi_R]$ is a bijection.
\end{corollary}

\begin{proof}
By Theorem \ref{thm:main5} and Proposition \ref{prop:normalized}.
\end{proof}

In the next section we will consider the inverse of the bijection 
in Corollary \ref{cor:main2}.

\section{The inverse of the bijection in Corollary \ref{cor:main2} }
\label{sec:known}

In this section we describe the inverse of the bijection in Corollary \ref{cor:main2}.
Let  $\Phi = (\{E_i\}_{i=0}^d; \{E^*_i\}_{i=0}^d)$ denote a symmetric idempotent system in $\A$,
and let the antiautomorphism $\dagger$ of $\A$ be from Definition \ref{def:sym}.
Let the  algebra $\mathcal M$ be from Definition \ref{def:M}.

\begin{defi}   {\rm (See \cite[Definition 4.1]{NT:ips}.) }
\label{def:mi}  \samepage
\ifDRAFT {\rm def:mi}. \fi
For $0 \leq i \leq d$ define
\begin{equation}
  m_i = \text{\rm tr} (E^*_0 E_i),               \label{eq:defmi}
\end{equation}
where {\rm tr} means trace.
\end{defi}

\begin{lemma}   {\rm (See \cite[Lemma 4.4]{NT:ips}.) }
\label{lem:mi}   \samepage
The following hold:
\begin{itemize}
\item[\rm (i)]
$m_i \neq 0 \quad (0 \leq i \leq d)$;
\item[\rm (ii)]
$\sum_{i=0}^d m_i = 1$.
\end{itemize}
\end{lemma}

\begin{defi}   {\rm (See \cite[Definition 4.5]{NT:ips}.) }
\label{def:ISnu}   \samepage
\ifDRAFT {\rm def:ISnu}. \fi
Setting $i=0$ in \eqref{eq:defmi} we find that $m_0 = m^*_0$;
let $\nu$ denote the multiplicative inverse of this common value.
We call $\nu$ the {\em size of $\Phi$}.
We emphasize that $\nu = \nu^*$.
\end{defi}

\begin{lemma}   {\rm (See \cite[Lemmas 6.3, 7.4]{NT:ips}.) }
\label{lem:Ai}    \samepage
\ifDRAFT {\rm lem:Ai}. \fi
For $0 \leq i \leq d$  there exists a unique $A_i \in {\mathcal M}$ such that
\[
    A_i E^*_0 E_0 = E^*_i E_0.   
\]
\end{lemma}
 
\begin{lemma}   {\rm (See \cite[Lemma 7.5]{NT:ips}.) }
\label{lem:A0}    \samepage
\ifDRAFT {\rm lem:A0}. \fi
We have $A_0 = I$.
\end{lemma}

\begin{lemma}    {\rm (See \cite[Lemma 7.7]{NT:ips}.) }
\label{lem:Aibasis}   \samepage
\ifDRAFT {\rm lem:Aibasis}. \fi
The elements $\{A_i\}_{i=0}^d$ form a basis for the vector space $\mathcal M$.
\end{lemma}

By Lemma \ref{lem:Aibasis} there exist scalars $p^h_{i j}$ $(0 \leq h, i, j \leq d)$ in $\F$
such that
\begin{align}
  A_i A_j &= \sum_{h=0}^d p^h_{i j} A_h   &&   (0 \leq i,j \leq d).      \label{eq:ISphij}
\end{align}

\begin{defi}    \label{def:ISki}    \samepage
\ifDRAFT {\rm def:ISki}. \fi
For $0 \leq i \leq d$ define $k_i = \nu m^*_i$.
\end{defi}

\begin{lemma}    {\rm (See \cite[Lemma 8.4]{NT:ips}.) }
\label{lem:ISki}    \samepage
\ifDRAFT {\rm lem:ISki}. \fi
The following hold:
\begin{itemize}
\item[\rm (i)]
$k_i \neq 0 \quad (0 \leq i \leq d)$;
\item[\rm (ii)]
$\nu = \sum_{i=0}^d k_i$;
\item[\rm (iii)]
$k_0 = 1$.
\end{itemize}
\end{lemma}

\begin{lemma}   {\rm (See \cite[Lemma 10.9]{NT:ips}.) }
\label{lem:ISp0ij}    \samepage
\ifDRAFT {\rm lem:ISp0ij}. \fi
For $0 \leq i,j \leq d$, $p^0_{i j} = \delta_{i,j} k_i$.
\end{lemma}

\begin{lemma}   {\rm (See \cite[Lemma 10.10]{NT:ips}.) }
\label{lem:ISkikj}    \samepage
\ifDRAFT {\rm lem:ISkikj}. \fi
For $0 \leq i,j \leq d$,  $k_i k_j = \sum_{h=0}^d p^h_{i j} k_h$.
\end{lemma}

\begin{defi}    \label{def:P}    \samepage
\ifDRAFT {\rm def:P}. \fi
Each of $\{E_i\}_{i=0}^d$ and $\{A_i\}_{i=0}^d$ is a basis for the
vector space $\mathcal M$.
Let $P = P_\Phi$ denote the transition matrix  from $\{E_i\}_{i=0}^d$ to $\{A_i\}_{i=0}^d$.
We call $P$ the {\em first eigenmatrix of $\Phi$}.
Let $Q = Q_\Phi$ denote the first eigenmatrix of $\Phi^*$.
We call $Q$ the {\em second eigenmatrix of $\Phi$}.
\end{defi}

\begin{lemma}    {\rm (See \cite[Lemmas 12.7, 12.8]{NT:ips}.) }
\label{lem:P0i}    \samepage
\ifDRAFT {\rm lem:P0i}. \fi
For $0 \leq i,j \leq d$ the following hold:
\begin{itemize}
\item[\rm (i)]
$P_{i,0} = 1$;
\item[\rm (ii)]
$P_{0,j} = k_j$;
\item[\rm (iii)]
$(P^{-1})_{i,0} = \nu^{-1}$;
\item[\rm (iv)]
$(P^{-1})_{0,j} = \nu^{-1} k^*_j$.
\end{itemize}
\end{lemma}

\begin{defi}   {\rm (See \cite[Definition 14.1]{NT:ips}.) }
\label{def:KKs}    \samepage
\ifDRAFT {\rm def:KKs}. \fi
Let $K$ (resp.\ $K^*$) denote the diagonal matrix in $\Matd$
that has $(i,i)$-entry $k_i$ (resp.\ $k^*_i$) for $0 \leq i \leq d$.
\end{defi}

\begin{lemma}   {\rm (See \cite[Lemma 14.2]{NT:ips}.) }
 \label{lem:PQ}    \samepage
\ifDRAFT {\rm lem:PQ}. \fi
The following hold:
\begin{itemize}
\item[\rm (i)]
$P Q = \nu I$;
\item[\rm (ii)]
$P^{\sf t} K^* = K Q$.
\end{itemize}
\end{lemma}

\begin{lemma}   {\rm (See \cite[Lemma 14.4, Proposition 18.1]{NT:ips}.) }
\label{lem:EiEsi}      \samepage
\ifDRAFT {\rm lem:EiEsi}. \fi
There exists an algebra isomorphism ${\mathcal A} \to \Matd$
that sends $E_i \mapsto \Delta_{i,i}$ and $E^*_i \mapsto P \Delta_{i,i} P^{-1}$
for $0 \leq i \leq d$.
\end{lemma}

\begin{corollary}    \label{cor:canonical}    \samepage
\ifDRAFT {\rm cor:canonical}. \fi
The first eigenmatrix $P$ is AO normalized solid invertible.
Moreover the idempotent system $\Phi_P$ is isomorphic to $\Phi$.
\end{corollary}

\begin{proof}
By Lemma \ref{lem:P0i}, $P$ is normalized solid invertible.
By Lemma \ref{lem:PQ}, $P$ is AO.
By Lemma \ref{lem:EiEsi}, $\Phi_P$ is isomorphic to $\Phi$.
\end{proof}

\begin{prop}     \label{prop:inverse}    \samepage
\ifDRAFT {\rm prop:inverse}. \fi
Referring to the bijection in Corollary \ref{cor:main2},
the inverse bijection sends $[\Phi] \mapsto P_\Phi$.
\end{prop}

\begin{proof}
By Corollary \ref{cor:canonical}.
\end{proof}

\begin{corollary}    \label{cor:uniquePhi}    \samepage
\ifDRAFT {\rm cor:uniquePhi}. \fi
Two symmetric idempotent systems over $\F$ are isomorphic
if and only if they have the same first eigenmatrix.
\end{corollary}

\begin{proof}
By Proposition \ref{prop:inverse}.
\end{proof}

\section{Character algebras}
\label{sec:SCalgebra}

Our next goal is to explain how AO normalized solid invertible matrices and
symmetric idempotent systems are related to
character algebras.
Traditionally a character algebra is defined over the complex number field
\cite{BI, Kawada}.
In the present paper we define a character algebra over an arbitrary field.

\begin{defi}   {\rm (See \cite[Section II.2.5]{BI}.)}
\label{def:SCalgebra}    \samepage
\ifDRAFT {\rm def:SCalgebra}. \fi
By a {\em character algebra over $\F$ with diameter $d$} we mean a sequence
\[
  ({\mathcal C}; \{ x_i \}_{i=0}^d),
\]
where $\mathcal C$ is a commutative $\F$-algebra,
and $\{x_i\}_{i=0}^d$ are elements in $\mathcal C$
that satisfy the following {\rm (i)--(iv)}.
\begin{itemize}
\item[(i)]
$x_0 = 1$.
\item[(ii)]
$\{x_i\}_{i=0}^d$ is a basis of the vector space $\mathcal C$.
\item[(iii)]
Define scalars $p^h_{i j}$ $(0 \leq h,i,j \leq d)$  such that
\begin{align}
  x_i x_j &= \sum_{h=0}^d p^h_{i j} x_h  && (0 \leq i,j \leq d).    \label{eq:xixj}
\end{align}
Then there exist nonzero  scalars $\{k_i\}_{i=0}^d$  such that
\begin{align}
  p^0_{i j} &= \delta_{i,j} k_i  &&  (0 \leq i,j \leq d).       \label{eq:p0ij}
\end{align}
\item[\rm (iv)]
There exists an algebra homomorphism $\vphi : {\mathcal C} \to \F$
such that $\vphi(x_i)= k_i$ for $0 \leq i \leq d$.
\end{itemize}
For historical reasons, we call the scalars $p^h_{i j}$ the {\em intersection numbers}.
\end{defi}

We refer the reader to \cite{Arad, BI, Kawada, Blau, Blau2, Egge, Pascasio} for background information on
character algebras.

Next we discuss the notion of isomorphism for character algebras.
Consider two character algebras 
$({\mathcal C}; \{x_i\}_{i=0}^d)$ and $({\mathcal C}' ; \{x'_i\}_{i=0}^d)$ over  $\F$.
By an {\em isomorphism of character algebras from $({\mathcal C}; \{x_i\}_{i=0}^d)$
to $({\mathcal C}' ; \{x'_i\}_{i=0}^d)$} we mean an algebra isomorphism 
${\mathcal C} \to {\mathcal C}'$ that sends $x_i \mapsto x'_i$ for $0 \leq i \leq d$.
The character algebras $({\mathcal C}; \{x_i\}_{i=0}^d)$ and $({\mathcal C}' ; \{x'_i\}_{i=0}^d)$
are said to be {\em isomorphic} whenever there exists an isomorphism of character algebras
from $({\mathcal C}; \{x_i\}_{i=0}^d)$ to $({\mathcal C}' ; \{x'_i\}_{i=0}^d)$.

\begin{lemma}   \label{lem:iso}    \samepage
\ifDRAFT {\rm lem:iso}. \fi
Two character algebras over $\F$ are isomorphic
if and only if they have the same intersection numbers.
\end{lemma}

\begin{proof}
Use \eqref{eq:xixj}.
\end{proof}

\begin{lemma}    \label{lem:Calgp}    \samepage
\ifDRAFT {\rm lem:Calgp}. \fi
Referring to the character algebra in Definition \ref{def:SCalgebra},
the intersection numbers  satisfy the following {\rm (i)--(iii)}:
\begin{itemize}
\item[\rm (i)]
$k_0 = 1$;
\item[\rm (ii)]
$p^h_{i 0}  = \delta_{h,i}$ for $0 \leq h,i \leq d$;
\item[\rm (iii)]
$p^h_{i j} = p^h_{j i}$ for $0 \leq h,i,j \leq d$.
\end{itemize}
\end{lemma}

\begin{proof}
(i), (ii)
Since $x_0 = 1$.

(iii)
Since $\mathcal C$ is commutative.
\end{proof}

As an illustration, we describe the character algebras of diameter $d=1$.

\begin{lemma}    \label{lem:ex1}   \samepage
\ifDRAFT {\rm lem:ex1}. \fi
For a character algebra $({\mathcal C}; \{x_i\}_{i=0}^1)$ over $\F$,
the intersection numbers satisfy the following:
\begin{itemize}
\item[\rm (i)]
$p^1_{1 1} = k_1 -1$;
\item[\rm (ii)]
$x_1^2 = k_1 x_0 + (k_1 -1) x_1$;
\item[\rm (iii)]
$e^2 = (k_1 + 1)e$, where $e=x_0 + x_1$.
\end{itemize}
\end{lemma}

\begin{proof}
By \eqref{eq:p0ij}, $p^0_{1 1} = k_1$.
By this and \eqref{eq:xixj},
\begin{equation}
 x_1^2 = k_1 x_0 + p^1_{1 1} x_1.    \label{eq:x12b}
\end{equation}
In this equation, apply $\vphi$ to each side to get
$k_1^2 = k_1 k_0 + p^1_{1 1} k_1$.
By this and since $k_0 = 1$, $k_1 \neq 0$ we get (i). 
By (i) and \eqref{eq:x12b} we get (ii). 
Using (ii) we get (iii).
\end{proof}

In the next result, we classify up to isomorphism the character algebras with diameter $d=1$.

\begin{prop}   \label{prop:ex2}    \samepage
\ifDRAFT {\rm prop:ex2}. \fi
Let $0 \neq k \in \F$.
Then up to isomorphism there exists a unique character algebra over $\F$
that has diameter $d=1$ and $k_1 = k$.
\end{prop}

\begin{proof}
First we show the uniqueness.
By \eqref{eq:p0ij} and Lemma \ref{lem:Calgp},
we find that all the intersection numbers are uniquely determined 
by $k_1$ and $p^1_{1 1}$.
By Lemma \ref{lem:ex1}(i), $p^1_{1 1}$ is determined by $k_1$.
By these comments, all the intersection numbers are determined by $k_1$.
Now the uniqueness follows by Lemma \ref{lem:iso}.
Next we show the existence.
Consider the quotient algebra ${\mathcal C} = \F[x]/{\mathcal I}$,
where $\F[x]$ is the $\F$-algebra of polynomials in a variable $x$, and
$\mathcal I$ is the ideal of $\F[x]$ generated by $(x+1)(x-k)$.
Define $x_0 = 1 + {\mathcal I}$ and $x_1 = x + {\mathcal I}$.
We show that $({\mathcal C}; \{x_i\}_{i=0}^1)$ is a character algebra.
To do this, we verify conditions (i)--(iv) in Definition \ref{def:SCalgebra}.
Condition (i) holds since by construction $x_0$ is the multiplicative identity
in $\mathcal C$.
Condition (ii) holds since by construction $x_0$, $x_1$ form a basis for $\mathcal C$.
Now consider condition (iii).
We just mentioned that $x_0$ is the identity in $\mathcal C$.
By construction $x_1^2 = k x_0 + (k-1) x_1$.
By these comments,
\begin{align*}
p^0_{0 0} &= 1,   &  p^0_{0 1} &= p^0_{1 0} = 0,  & p^0_{1 1} &= k.
\end{align*}
Thus  \eqref{eq:p0ij} holds with  $k_0 = 1$, $k_1 = k$.
We have verified condition (iii).
Concerning condition (iv),
note that $k$ is a root of the polynomial $(x+1)(x-k)$,
so there exists an algebra homomorphism
$\vphi: {\mathcal C} \to \F$ that sends $x_1$ to $k$.
The map $\vphi$ satisfies the requirements of condition (iv).
We have shown that $({\mathcal C}; \{x_i\}_{i=0}^1)$ is a character algebra.
By construction this character algebra has $k_1 = k$.
\end{proof}

\section{Semisimple character algebras and character systems}
\label{sec:Csystem}

In this section we discuss a type of character algebra, said to be semisimple.
Motivated by this type of character algebra, we introduce the notion of a 
character system.

\begin{defi}   \label{def:ss}    \samepage
\ifDRAFT {\rm def:ss}. \fi
A character algebra $({\mathcal C}; \{ x_i \}_{i=0}^d)$ is said to be {\em semisimple}
whenever there exists a basis  $\{e_i\}_{i=0}^d$ of the vector space $\mathcal C$ such that
$e_i e_j = \delta_{i,j} e_i$ $(0 \leq i,j \leq d)$ and $1 = \sum_{i=0}^d e_i$.
In this case, $\{e_i\}_{i=0}^d$ are unique up to permutation, and called the {\em primitive idempotents}
of $\mathcal C$.
\end{defi}

We are mainly interested in the semisimple character algebras.

\begin{note}   \label{note:BI}    \samepage
\ifDRAFT {\rm note:BI}. \fi
Consider a character algebra from Definition \ref{def:SCalgebra}.
In \cite[Section II.2.5]{BI}, the intersection numbers are assumed to be real
with $k_i > 0$ $(0 \leq i \leq d)$,
and under this assumption, it is shown that the character algebra is semisimple \cite[Proposition II.5.4]{BI}.
\end{note}

\begin{lemma}   \label{lem:Cei}   \samepage
\ifDRAFT {\rm lem:Cei}. \fi
Let $({\mathcal C}; \{x_i\}_{i=0}^d)$ denote a semisimple character algebra over $\F$.
Then there exists a unique primitive idempotent of ${\mathcal C}$
that is not sent to zero by $\vphi$.
This primitive idempotent is sent to $1$ by $\vphi$.
\end{lemma}

\begin{proof}
Concerning the existence, observe that the sum of the primitive idempotents of $\mathcal C$
is equal to the multiplicative identity of $\mathcal C$,
and that $\vphi$ sends this identity to $1$.
Concerning the uniqueness, 
let $e_0$ denote a primitive idempotent of ${\mathcal C}$ such that $\vphi (e_0) \neq 0$.
Pick any other primitive idempotent $f$ in $\mathcal C$.
We have $e_0 f = 0$, so
$0 = \vphi (e_0 f) = \vphi (e_0) \vphi (f)$.
By this and  $\vphi (e_0) \neq 0$ we get $\vphi (f) = 0$.
We have shown the uniqueness.
The last assertion of the lemma statement follows from our above remarks.
\end{proof}

\begin{defi}    \label{def:e0}    \samepage
\ifDRAFT {\rm def:e0}. \fi
Referring to Lemma \ref{lem:Cei},
let $e_0$ denote the unique primitive idempotent of $\mathcal C$
that is not sent to $0$ by $\vphi$.
The primitive idempotent $e_0$ is said to be {\em trivial}.
Note that $\vphi (e_0) = 1$.
\end{defi}

In Proposition \ref{prop:ex2} we described the character algebras with diameter $d=1$.
In the next result we determine which of these character algebras is semisimple.

\begin{lemma}    \label{lem:ex3}    \samepage
\ifDRAFT {\rm lem:ex3}. \fi
Let $({\mathcal C}, \{x_i\}_{i=0}^1)$ denote the character algebra from Proposition \ref{prop:ex2}.
Then $({\mathcal C}, \{x_i\}_{i=0}^1)$ is semisimple if and only if $k \neq -1$.
In this case the primitive idempotents satisfy
\begin{align}
 x_0 &= e_0 + e_1, &
 x_1 &= k e_0 - e_1,                                      \label{eq:x0x1}
\\
 e_0 &= \frac{ x_0 + x_1} { k+1 }, &
 e_1 &= \frac{ k x_0 - x_1} {k+1}.                    \label{eq:e0e1}
\end{align}
\end{lemma}

\begin{proof}
To prove the lemma in one direction,
we assume that $k=-1$ and show that $({\mathcal C}; \{x_i\}_{i=0}^1)$  is not semisimple.
To show this,
we assume that $({\mathcal C}; \{x_i\}_{i=0}^1)$ is semisimple and get a contradiction.
By Lemma \ref{lem:ex1}(iii) we have $e^2 = 0$, where $e = x_0 + x_1$.
Write $e = \alpha_0 e_0 + \alpha_1 e_1$, where $\alpha_0$, $\alpha_1 \in \F$ and
$e_0$, $e_1$ are the primitive idempotents of $\mathcal C$.
Using $e^2=0$ we find $\alpha_0^2 = 0$ and $\alpha_1^2 = 0$,
forcing $\alpha_0 = \alpha_1 = 0$ so $e=0$, contradicting the fact
that $x_0$, $x_1$ are linearly independent.
We have proved the lemma in one direction.
To prove the lemma in the other direction,
we assume that $k \neq -1$ and show that $({\mathcal C}; \{x_i\}_{i=0}^1)$ is semisimple.
Define elements $\{e_i\}_{i=0}^1$ by \eqref{eq:e0e1}.
Using Lemma \ref{lem:ex1} with $k_1=k$ we find that $e_0$, $e_1$ form a basis for $\mathcal C$
such that 
\begin{align*}
  e_0 + e_1 &= 1,  &
  e_0^2 &= e_0, &
  e_0 e_1 &= 0, &
  e_1^2 &= e_1.
\end{align*}
By these comments $({\mathcal C}; \{x_i\}_{i=0}^1)$ is semisimple with primitive idempotents $e_0$, $e_1$.
Line \eqref{eq:x0x1} is obtained from \eqref{eq:e0e1}.
\end{proof}

\begin{defi}    \label{def:SCsystem}    \samepage
\ifDRAFT {\rm def:SCsystem}. \fi
By a {\em character system over $\F$ with diameter $d$}, we mean a sequence
\[
   \Psi =  ( {\mathcal C}; \{x_i\}_{i=0}^d; \{e_i\}_{i=0}^d),
\]
where $( {\mathcal C}; \{x_i\}_{i=0}^d)$ is a semisimple character algebra over $\F$,
and $\{e_i\}_{i=0}^d$ are the primitive idempotents of $\mathcal C$ with $e_0$ trivial.
We say that $( {\mathcal C}; \{x_i\}_{i=0}^d)$ and $\Psi$ are {\em associated}.
\end{defi}

Next we discuss the notion of isomorphism for character systems.
Suppose we are given two character systems over $\F$,
denoted $\Psi =  ( {\mathcal C}; \{x_i\}_{i=0}^d; \{e_i\}_{i=0}^d)$
and $\Psi' =  ( {\mathcal C}'; \{x'_i\}_{i=0}^d; \{e'_i\}_{i=0}^d)$.
By an {\em isomorphism of character systems} from $\Psi$ to $\Psi'$
we mean an algebra isomorphism ${\mathcal C} \to {\mathcal C}'$ 
that sends $x_i \mapsto x'_i$ and $e_i \mapsto e'_i$ for
$0 \leq i \leq d$.
We say $\Psi$ and $\Psi'$ are {\em isomorphic} whenever
there exists an isomorphism of character systems from $\Psi$ to $\Psi'$.

\begin{defi}     \label{def:[Psi]}    \samepage
\ifDRAFT {\rm def:[Psi]}. \fi
Referring to the character system $\Psi$ in Definition \ref{def:SCsystem},
let $[\Psi]$ denote the isomorphism class that contains $\Psi$.
\end{defi}

\begin{defi}    \label{def:CS}    \samepage
\ifDRAFT {\rm def:CS}. \fi
Let $\text{\rm CS}_d(\F)$ denote the set consisting of
the isomorphism classes of character systems over $\F$ with diameter $d$.
\end{defi}

\begin{defi}  {\rm (See \cite[p.\ 90]{BI}.) }
\label{def:P2}    \samepage
\ifDRAFT {\rm def:P2}. \fi
Let $\Psi =  ( {\mathcal C}; \{x_i\}_{i=0}^d; \{e_i\}_{i=0}^d)$ denote a character system over $\F$.
Let $P = P_\Psi$ denote the transition matrix from the basis $\{e_i\}_{i=0}^d$ of $\mathcal C$
to the basis $\{x_i\}_{i=0}^d$ of $\mathcal C$.
We call $P$ the {\em eigenmatrix of $\Psi$}.
\end{defi}

\begin{lemma}    \label{lem:d1P}    \samepage
\ifDRAFT {\rm lem:d1P}. \fi
For the character system associated with the semisimple character algebra in Lemma \ref{lem:ex3},
the eigenmatrix $P$ satisfies
\begin{align*}
 P &= \begin{pmatrix}
        1 & k  \\
        1 & -1
       \end{pmatrix},
&
 P^{-1}  &= \frac{1} {k+1} 
  \begin{pmatrix}
        1 & k  \\
        1 & -1
       \end{pmatrix}.
\end{align*}
\end{lemma}

\begin{proof}
By \eqref{eq:x0x1} and \eqref{eq:e0e1}.
\end{proof}

\begin{prop}    \label{prop:CalgP}    \samepage
\ifDRAFT {\rm prop:CalgP}. \fi
Two character systems over $\F$ are isomorphic if and only if they have
the same eigenmatrix.
\end{prop}

\begin{proof}
Let $\Psi$, $\Psi'$ denote the character systems in question.
First assume that $\Psi$ and $\Psi'$ are isomorphic.
Then $\Psi$ and $\Psi'$ have the same eigenmatrix by Definition \ref{def:P2}.
Next assume that $\Psi$ and $\Psi'$ have the same eigenmatrix.
Then $\Psi$ and $\Psi$ have the same diameter.
Write
$\Psi = ( {\mathcal C}; \{x_i\}_{i=0}^d; \{e_i\}_{i=0}^d)$
and $\Psi' =  ( {\mathcal C}'; \{x'_i\}_{i=0}^d; \{e'_i\}_{i=0}^d)$.
Consider the $\F$-linear map $\gamma : {\mathcal C} \to {\mathcal C}'$ that 
sends $e_i \mapsto e'_i$ for $0 \leq i \leq d$.
Then $\gamma$ is an algebra isomorphism.
By Definition \ref{def:P2},
$\gamma$ sends $x_i \mapsto x'_i$ for $0 \leq i \leq d$.
Thus $\gamma$ is an isomorphism of character systems from $\Psi$ to $\Psi'$.
Therefore the character systems $\Psi$ and $\Psi'$ are isomorphic.
\end{proof}

\begin{lemma}    \label{lem:xie02}    \samepage
\ifDRAFT {\rm lem:xie02}.  \fi
Let $\Psi =  ( {\mathcal C}; \{x_i\}_{i=0}^d; \{e_i\}_{i=0}^d)$ denote a character system over $\F$.
Then the eigenmatrix $P = P_\Psi$ satisfies the following {\rm (i), (ii)} for $0 \leq i,j \leq d$:
\begin{itemize}
\item[\rm (i)]
$P_{i, 0} = 1$;
\item[\rm (ii)]
$P_{0,j} = k_j$.
\end{itemize}
\end{lemma}

\begin{proof}
(i)
Since  $x_0 = 1 = \sum_{i=0}^d e_i$.

(ii)
We have $x_j = \sum_{i=0}^d P_{i,j} e_i$.
In this equation, apply $\vphi$ to each side.
By Definition \ref{def:SCalgebra}(iv) we have $\vphi (x_j ) = k_j$,
and by Definition \ref{def:e0} we have $\vphi (e_i) = \delta_{i,0}$ for $0 \leq i \leq d$.
By these comments we get the result.
\end{proof}

\section{Some properties of the eigenmatrix of a character system}
\label{sec:eigen}

In this section we show that the eigenmatrix of a character system
is AO normalized solid invertible.
Throughout this section 
let $ ({\mathcal C}; \{x_i\}_{i=0}^d)$ denote a character algebra over $\F$.
The following definition is a variation on \cite[p.\ 145]{AB}.

\begin{defi}    \label{def:bilin}    \samepage
\ifDRAFT {\rm def:bilin}. \fi
Define a bilinear form $\b{ \; , \; } : {\mathcal C} \times {\mathcal C} \to \F$
such that
\begin{align}
 \b{ x_i, x_j } &= \delta_{i,j} k_i   &&  (0 \leq i,j \leq d).       \label{eq:bxixj}
\end{align}
\end{defi}

\begin{lemma}   \label{lem:nondeg}    \samepage
\ifDRAFT {\rm lem:nondeg}. \fi
The bilinear form $\b{ \; , \;}$ is symmetric and nondegenerate.
\end{lemma}

\begin{proof}
By \eqref{eq:bxixj} and since $k_i \neq 0$ for $0 \leq i \leq d$.
\end{proof}

The following result is a variation on \cite[Proposition 2.5]{Blau}.

\begin{lemma}    \label{lem:buv}    \samepage
\ifDRAFT {\rm lem:buv}. \fi
For $u, v \in {\mathcal C}$,
$\b{ u, v } = \b{ u v, x_0}$.
\end{lemma}

\begin{proof}
Without loss of generality, we assume that  $u = x_i$ and $v= x_j$ for some 
integers $i$, $j$ $(0 \leq i,j \leq d)$.
Using \eqref{eq:xixj}, \eqref{eq:bxixj}, \eqref{eq:p0ij}, \eqref{eq:bxixj} in order,
\[
 \b{ x_i x_j, x_0 } = \sum_{h=0}^d p^h_{i j} \b{x_h, x_0} = \sum_{h=0}^d p^h_{i j} \delta_{h,0} k_0
 = p^0_{i j} = \delta_{i,j} k_i = \b{x_i, x_j}.
\]
The result follows.
\end{proof}

For the rest of this section, assume that $({\mathcal C}; \{x_i\}_{i=0}^d)$ is semisimple,
and let $\Psi = ({\mathcal C}; \{x_i\}_{i=0}^d; \{e_i\}_{i=0}^d)$ denote an associated character system.
Recall the eigenmatrix  $P = P_\Psi$ from Definition \ref{def:P2}.

\begin{lemma}    \label{lem:beiej}    \samepage
\ifDRAFT {\rm lem:beiej}. \fi
The basis $\{e_i\}_{i=0}^d$ of  $\mathcal C$ is orthogonal with respect to $\b{ \; , \;}$.
\end{lemma}

\begin{proof}
For $0 \leq i,j \leq d$ with $i \neq j$ we have $e_i e_j = 0$.
By this and Lemma \ref{lem:buv} we have
$\b{e_i, e_j} = \b{e_i e_j, x_0} = 0$.
The result follows.
\end{proof}

\begin{defi}    \label{def:nu1}   \samepage
\ifDRAFT {\rm def:nu1}. \fi
For $0 \leq i \leq d$ define $m_i = \b{e_i, e_i}$ and
note that $m_i$ is nonzero.
Define $\nu = m_0^{-1}$.
We call $\nu$ the {\em size of $\Psi$}.
\end{defi}

\begin{defi}    \label{def:nu2}   \samepage
\ifDRAFT {\rm def:nu2}. \fi
For $0 \leq i \leq d$ define $k^*_i = \nu m_i$.
By construction $k^*_0 = 1$ and
\begin{align}
  \b{e_i, e_i} &= \nu^{-1} k^*_i   &&  (0 \leq i \leq d).    \label{eq:beiei}
\end{align}
\end{defi}

\begin{lemma}    \label{lem:beixj}     \samepage
\ifDRAFT {\rm lem:beixj}. \fi
For $0 \leq i,j \leq d$,
\begin{equation}
\b{ x_i, e_j } = \nu^{-1} P_{j,i} k^*_j = k_i (P^{-1})_{i,j}.          \label{eq:beixj}
\end{equation}
\end{lemma}

\begin{proof}
We have $\b{x_i ,e_j} = P_{j,i} \b{e_j, e_j}$ by Lemma \ref{lem:beiej} and 
since $P$ is the transition matrix from $\{e_\ell\}_{\ell=0}^d$ to $\{x_\ell\}_{\ell=0}^d$.
We have $\b{x_i, e_j} = (P^{-1})_{i,j} \b{x_i, x_i}$ by \eqref{eq:bxixj} and since
$P^{-1}$ is the transition matrix from $\{x_\ell\}_{\ell=0}^d$ to $\{e_\ell\}_{\ell=0}^d$.
By these comments and  \eqref{eq:bxixj}, \eqref{eq:beiei} we get the result.
\end{proof}

\begin{lemma}    \label{lem:bxe0}    \samepage
\ifDRAFT {\rm lem:bxe0}. \fi
For $x \in {\mathcal C}$ we have $\vphi(x) = \nu \b{x, e_0}$.
\end{lemma}

\begin{proof}
Write $x = \sum_{i=0}^d \alpha_i e_i$ with $\alpha_i \in \F$ for $0 \leq i \leq d$.
By Definition \ref{def:e0} we get $\vphi (x) = \alpha_0$.
Using Lemma \ref{lem:beiej} and Definition \ref{def:nu1}, 
$\b{x, e_0} = \alpha_0 \b{e_0, e_0} = \alpha_0 \nu^{-1}$.
By these comments we get the result.
\end{proof}

\begin{lemma}   \label{lem:sumxi}   \samepage
\ifDRAFT {\rm lem:sumxi}. \fi
We have $ \sum_{i=0}^d x_i = \nu e_0$.
\end{lemma}

\begin{proof}
Write $e_0 = \sum_{\ell=0}^d \alpha_\ell x_\ell$ with $\alpha_\ell \in \F$ for $0 \leq \ell \leq d$.
Pick an integer $i$ $(0 \leq i \leq d)$.
In the previous equation, take the inner product with $x_i$,
and simplify the result using \eqref{eq:bxixj} to get $\b{x_i, e_0} = \alpha_i k_i$.
By Definition \ref{def:SCalgebra}(iv) and Lemma \ref{lem:bxe0}, 
$k_i = \vphi (x_i) = \nu \b{ x_i, e_0}$.
By these comments, $\alpha_i = \nu^{-1}$.
The result follows.
\end{proof}

\begin{lemma}   \label{lem:PAOpre}   \samepage
\ifDRAFT {l\rm em:PAOpre}. \fi
For $0 \leq i,j \leq d$ the following hold:
\begin{itemize}
\item[\rm (i)]
$(P^{-1})_{i,0} = \nu^{-1}$;
\item[\rm (ii)]
$(P^{-1})_{0,j} = \nu^{-1} k^*_j$.
\end{itemize}
\end{lemma}

\begin{proof}
(i)
Set $j=0$ in \eqref{eq:beixj}  and evaluate the result using $k^*_0=1$ along with
Lemma \ref{lem:xie02}(ii).

(ii)
Set $i=0$ in \eqref{eq:beixj}  and evaluate the result using
$k_0=1$ along with Lemma \ref{lem:xie02}(i).
\end{proof}

The following two propositions are variations on \cite[Theorem II.5.5]{BI}.

\begin{prop}    \label{prop:Psolid}    \samepage
\ifDRAFT {\rm prop:Psolid}. \fi
The eigenmatrix $P$ of $\Psi$ is normalized solid invertible.
\end{prop}

\begin{proof}
By Lemmas \ref{lem:xie02} and \ref{lem:PAOpre}.
\end{proof}

\begin{prop}   \label{prop:PAO}    \samepage
\ifDRAFT {\rm prop:PAO}. \fi
The eigenmatrix $P$ of $\Psi$ is AO.
\end{prop}

\begin{proof}
By Definition \ref{def:feasible} it suffices to show that
$P^{-1}$ and $P^{\sf t}$ are diagonally equivalent.
Let $K$ (resp.\ $K^*$) denote the diagonal matrix in $\Matd$ that has
$(i,i)$-entry $k_i$ (resp.\ $k^*_i$) for $0 \leq i \leq d$.
Note that $K$ and $K^*$ are invertible.
By \eqref{eq:beixj} we have $\nu^{-1} P^{\sf t} K^* = K P^{-1}$,
and so  $P^{-1} = \nu^{-1} K^{-1} P^{\sf t} K^*$.
Thus $P^{-1}$ and $P^{\sf t}$ are diagonally equivalent.
The result follows.
\end{proof}

By Propositions \ref{prop:CalgP}, \ref{prop:Psolid}, \ref{prop:PAO}
there exists a map
$\text{\rm CS}_d (\F) \to \text{\rm AON}_d (\F)$
that sends $[\Psi] \mapsto P_\Psi$.
In the next section we will show that this map is a bijection.

\section{Symmetric idempotent systems, character systems, and
AO normalized solid invertible matrices}
\label{sec:SISCS}

Recall the sets $\text{\rm AON}_d (\F)$, $\text{\rm SIS}_d (\F)$
from Definition \ref{def:AONSIPS}
and the set $\text{\rm CS}_d (\F)$ from Definition \ref{def:CS}.
As we mentioned in Section \ref{sec:intro},
our goal is to show that these three sets are mutually in bijection,
and to describe the bijections involved.
In Sections \ref{sec:normalized} and \ref{sec:known},
we obtained a bijection $\text{\rm AON}_d (\F) \to \text{\rm SIS}_d (\F)$,
and described its inverse.
At the end of Section \ref{sec:eigen} we obtained a map
$\text{\rm CS}_d (\F) \to \text{\rm AON}_d (\F)$.
In the present section we show that this map is a bijection,
and describe its inverse.
We also describe the bijective correspondence between 
$\text{\rm SIS}_d (\F)$ and $\text{\rm CS}_d (\F)$.

\begin{prop}     \label{prop:PsiPhi}    \samepage
\ifDRAFT {\rm prop:PsiPhi}. \fi
Let $\Phi = (\{E_i\}_{i=0}^d; \{E^*_i\}_{i=0}^d)$ denote a symmetric idempotent system over $\F$.
Define a sequence
\[
  \Psi_\Phi =  ({\mathcal M}; \{A_i\}_{i=0}^d; \{E_i\}_{i=0}^d),
\]
where the algebra $\mathcal M$ is from Definition \ref{def:M}
and the elements  $\{A_i\}_{i=0}^d$ are from Lemma \ref{lem:Ai}.
Then $\Psi_\Phi$ is a character system over $\F$ 
whose eigenmatrix is the first eigenmatrix of $\Phi$.
\end{prop}

\begin{proof}
We first show that $({\mathcal M}; \{A_i\}_{i=0}^d)$ is a character algebra over $\F$.
We verify the conditions (i)--(iv) in Definition \ref{def:SCalgebra}.
Condition (i) holds by Lemma \ref{lem:A0}, 
and condition (ii) holds by Lemma \ref{lem:Aibasis}.
Recall the scalars $\{k_i\}_{i=0}^d$ from Definition \ref{def:ISki}
and the  scalars $p^h_{i j}$ from  \eqref{eq:ISphij}.
Condition (iii) holds by Lemmas \ref{lem:ISki}(i) and \ref{lem:ISp0ij}.
The elements $\{E_i\}_{i=0}^d$ form a basis of $\mathcal M$, so there exists an
$\F$-linear map $\vphi : {\mathcal M} \to \F$ such that
$\vphi (E_i) = \delta_{i,0}$ for $0 \leq i \leq d$.
One routinely checks that $\vphi$ is an algebra homomorphism.
Recall the first eigenmatrix $P = P_\Phi$ from Definition \ref{def:P}.
By Definition \ref{def:P},
$A_j = \sum_{i=0}^d P_{i,j} E_i$ for $0 \leq j \leq d$.
In this equation,  apply $\vphi$ to each side to find that
$\vphi (A_j ) = P_{0,j}$ for $0 \leq j \leq d$.
By this and Lemma \ref{lem:P0i}(ii),  $\vphi (A_j) = k_j$ for $0 \leq j \leq d$.
Thus condition (iv) holds.
We have shown that $({\mathcal M}; \{A_i\}_{i=0}^d)$ is a character algebra over $\F$.
By construction, $\{E_i\}_{i=0}^d$ are the primitive idempotents of $\mathcal M$.
So $({\mathcal M}; \{A_i\}_{i=0}^d)$ is semisimple.
By construction $\vphi (E_i) = \delta_{i,0}$ for $0 \leq i \leq d$, so 
the primitive idempotent $E_0$ is trivial.
Now $\Psi_\Phi$ is a character system over $\F$, in view of Definition \ref{def:SCsystem}.
By Definitions \ref{def:P}, \ref{def:P2}, $P$ is the eigenmatrix of $\Psi_\Phi$.
The result follows.
\end{proof}

By Proposition \ref{prop:PsiPhi} together with
Corollary \ref{cor:uniquePhi} and Proposition \ref{prop:CalgP}, 
we have a map
$\text{\rm SIS}_d (\F) \to \text{\rm CS}_d (\F)$,
$[\Phi] \mapsto [\Psi_\Phi]$.

\begin{prop}    \label{prop:main4}    \samepage
\ifDRAFT {\rm prop:main4}. \fi
The following {\rm (i)--(iii)} hold.
\begin{itemize}
\item[\rm (i)]
The map $\text{\rm SIS}_d (\F) \to \text{\rm AON}_d (\F)$,
$[\Phi] \mapsto P_\Phi$
is equal to the composition
\[
\xymatrix@C=60pt{
   \text{\rm SIS}_d (\F)
  \ar[r]_{ [\Phi] \; \mapsto \; [\Psi_\Phi] }
&
   \text{\rm CS}_d (\F)
  \ar[r]_{ [\Psi] \; \mapsto \; P_\Psi }
&
  \text{\rm AON}_d (\F).
}
\]
\item[\rm (ii)]
The map $\text{\rm CS}_d (\F) \to \text{\rm AON}_d (\F)$,
$[\Psi] \mapsto P_\Psi$
is a bijection.
\item[\rm (iii)]
The map $\text{\rm SIS}_d (\F) \to \text{\rm CS}_d (\F)$,
$[\Phi] \mapsto [\Psi_\Phi]$
is a bijection.
\end{itemize}
\end{prop}

\begin{proof}
(i) 
By the last assertion of Proposition \ref{prop:PsiPhi}.

(ii)
By Proposition \ref{prop:inverse}, the map
$\text{\rm SIS}_d (\F) \to \text{\rm AON}_d (\F)$,
$[\Phi] \mapsto P_\Phi$
is bijective.
By this and (i) above,
the map 
$\text{\rm CS}_d (\F) \to \text{\rm AON}_d (\F)$,
$[\Psi] \mapsto P_\Psi$
is surjective.
By Proposition \ref{prop:CalgP}
the map $\text{\rm CS}_d (\F) \to \text{\rm AON}_d (\F)$,
$[\Psi] \mapsto P_\Psi$  is injective.
The result follows.

(iii)
By (i), (ii) above and since the map
$\text{\rm SIS}_d (\F) \to \text{\rm AON}_d (\F)$,
$[\Phi] \mapsto P_\Phi$ is a bijection.
\end{proof}

Our next goal is to describe the inverse of the bijection in Proposition \ref{prop:main4}(ii).

\begin{lemma}    \label{lem:PsiP}   \samepage
\ifDRAFT {\rm lem:PsiP}. \fi
Let $P$ denote an AO normalized solid invertible matrix in $\Matd$.
Let $\{e_i\}_{i=0}^d$ denote indeterminates,
and let $\mathcal C$ denote the vector space over $\F$ with basis $\{e_i\}_{i=0}^d$.
Turn $\mathcal C$ into an algebra such that
$e_i e_j = \delta_{i,j} e_i$ for $0 \leq i,j \leq d$.
View $P$ as the transition matrix from the basis $\{e_i\}_{i=0}^d$ of $\mathcal C$
to a basis $\{x_i\}_{i=0}^d$ of $\mathcal C$.
Then
\[
   \Psi_P  = ({\mathcal C} ; \{x_i\}_{i=0}^d; \{e_i\}_{i=0}^d)
\]
is a character system over $\F$ that has eigenmatrix $P$.
\end{lemma}

\begin{proof}
Recall the idempotent system $\Phi = \Phi_P$ from Definition \ref{def:PhiR}
and the character system $\Psi_\Phi$ from Proposition \ref{prop:PsiPhi}.
Note by Corollary \ref{cor:canonical} that $\Phi$ has first eigenmatrix $P$.
By this and Proposition \ref{prop:PsiPhi},
$\Psi_\Phi$ has eigenmatrix $P$.
By construction
\[
  \Psi_\Phi = ({\mathcal D}; \{A_i\}_{i=0}^d; \{E_i\}_{i=0}^d),
\]
where $E_i = \Delta_{i,i}$ for $0 \leq i \leq d$,
and $\{A_i\}_{i=0}^d$ are from Lemma \ref{lem:Ai}.
Consider the $\F$-linear map $\pi : {\mathcal C} \to {\mathcal D}$
that sends $e_i \mapsto E_i$ for $0 \leq i \leq d$.
By construction, $\pi$ is an algebra isomorphism.
We have $\pi ( x_i ) = A_i$ for $0 \leq i \leq d$,
since $P$ is the transition matrix from $\{e_i\}_{i=0}^d$ to $\{x_i\}_{i=0}^d$
and also the transition matrix from $\{E_i\}_{i=0}^d$ to $\{A_i\}_{i=0}^d$.
Therefore $\Psi_P$ is a character system over $\F$,
and $\pi$ is an isomorphism of character systems from $\Psi_P$
to $\Psi_\Phi$.
The result follows.
\end{proof}

\begin{prop}    \label{prop:CSAON}    \samepage
\ifDRAFT {\rm prop:CSAON}. \fi
The following maps are inverses:
\begin{itemize}
\item[\rm (i)]
$\text{\rm AON}_d (\F) \to \text{\rm CS}_d (\F)$,
$P \mapsto [\Psi_P]$;
\item[\rm (ii)]
$\text{\rm CS}_d (\F) \to \text{\rm AON}_d (\F)$,
$[\Psi] \mapsto P_\Psi$.
\end{itemize}
\end{prop}

\begin{proof}
By Proposition \ref{prop:main4} the map (ii) is a bijection.
Pick any AO normalized solid invertible $P \in \Matd$.
By Lemma \ref{lem:PsiP}, $\Psi_P$ has eigenmatrix $P$.
So the composition of  (i) and (ii) is the identity.
By these comments we get the result.
\end{proof}

Our next goal is to describe the inverse of the bijection in Proposition \ref{prop:main4}(iii).

\begin{lemma}    \label{lem:PhiPsi}     \samepage
\ifDRAFT {\rm lem:PhiPsi}. \fi
Let $\Psi = ({\mathcal C} ; \{x_i\}_{i=0}^d; \{e_i\}_{i=0}^d)$ denote a character system
over $\F$.
For $0 \leq i \leq d$ define  $E_i$, $E^*_i \in \text{\rm End}({\mathcal C})$
such that
\begin{align}
  E_i e_j &= \delta_{i,j} e_j,  & 
 E^*_i x_j &= \delta_{i,j} x_j      \label{eq:EiEsidef}
\end{align}
for $0 \leq j \leq d$.
Then the sequence
\[
   \Phi_\Psi = (\{E_i\}_{i=0}^d; \{E^*_i\}_{i=0}^d)
\]
is a symmetric idempotent system in $\text{\rm End} ({\mathcal C})$
whose first eigenmatrix is the eigenmatrix of $\Psi$.
\end{lemma}

\begin{proof}
Recall that $P=P_\Psi$ is the transition matrix from the basis $\{e_i\}_{i=0}^d$ of $\mathcal C$
to the basis $\{x_i\}_{i=0}^d$ of $\mathcal C$.
By Propositions \ref{prop:Psolid}, \ref{prop:PAO}, $P$ is AO normalized solid invertible.
By this and Theorem \ref{thm:main5},
\[
  \Phi_P = ( \{\Delta_{i,i} \}_{i=0}^d; \{ P \Delta_{i,i} P^{-1} \}_{i=0}^d)
\]
is a symmetric idempotent system in $\Matd$.
We show that $\Phi_\Psi$ is a symmetric idempotent system in $\text{\rm End}({\mathbb C})$
that is isomorphic to $\Phi_P$.
Let $\pi : \text{\rm End}({\mathcal C}) \to \Matd$ denote the algebra isomorphism that
sends each $A \in \text{\rm End}({\mathcal C})$ to the matrix in $\Matd$ that represents $A$
with respect to the basis $\{e_i\}_{i=0}^d$.
By the equation on the left in \eqref{eq:EiEsidef}, $\pi (E_i) = \Delta_{i,i}$ for $0 \leq i \leq d$.
By the equation on the right in \eqref{eq:EiEsidef} along with linear algebra
we obtain $\pi (E^*_i) = P \Delta_{i,i} P^{-1}$ for $0 \leq i \leq d$.
By these comments, 
$\Phi_\Psi$ is an idempotent system in $\text{\rm End}( {\mathcal C})$,
and $\pi$ is an isomorphism of idempotent systems from $\Phi_\Psi$ to
$\Phi_P$.
The idempotent system $\Phi_\Psi$ is symmetric since $\Phi_P$ is symmetric.
By Corollary \ref{cor:canonical}, $\Phi_P$ has first eigenmatrix $P$.
Thus $\Phi_\Psi$ has first eigenmatrix $P$.
\end{proof}

By Corollary \ref{cor:uniquePhi}, Proposition \ref{prop:CalgP}, 
and Lemma \ref{lem:PhiPsi},
we have a map
$\text{\rm CS}_d (\F) \to \text{\rm SIS}_d (\F)$,
$[\Psi] \mapsto [\Phi_\Psi]$.

\begin{prop}    \label{prop:SISCS}     \samepage
\ifDRAFT {\rm cor:SISCS}. \fi
The following maps are inverses:
\begin{itemize}
\item[\rm (i)]
$\text{\rm SIS}_d (\F) \to \text{\rm CS}_d (\F)$,
 $[\Phi] \mapsto [ \Psi_\Phi ]$;
\item[\rm (ii)]
$\text{\rm CS}_d (\F) \to \text{\rm SIS}_d (\F)$,
$[\Psi] \mapsto [ \Phi_\Psi]$.
\end{itemize}
\end{prop}

\begin{proof}
By Corollary \ref{cor:uniquePhi} and Proposition \ref{prop:CalgP}.
\end{proof}

In summary, we have shown that the three sets
$\text{\rm AON}_d (\F)$,
$\text{\rm SIS}_d (\F)$,
$\text{\rm CS}_d (\F)$
are mutually in bijection,
and we described the bijections involved.
They are displayed in the following diagram.
\[
\xymatrix@C=60pt@R=60pt{
   \text{\rm AON}_d (\F)
  \ar[r]^{ \;\;  P \; \rightarrow \;  [ \Phi_P ]  }
  \ar[d]_{ \text{\rm \normalsize id} \; }
&
   \text{\rm SIS}_d (\F)
   \ar[l]^{ P_\Phi \; \leftarrow \; [\Phi]  }
  \ar[d]^-{ \footnotesize 
       \begin{matrix}  
         \left[ \Phi \right]  \rule{0mm}{2.3ex} \\ \downarrow \rule{0mm}{2.5ex} \\  [\Psi_\Phi] \rule{0mm}{2.5ex}
      \end{matrix} 
   }
\\
  \text{\rm AON}_d (\F)
  \ar[u]
  \ar[r]_{ \;\;\;\;  P \;\; \rightarrow \; [\Psi_P] }
&
  \text{\rm CS}_d (\F)
  \ar[u]^{   \footnotesize 
       \begin{matrix}  
         \left[ \Phi_\Psi \right]  \rule{0mm}{2.3ex} \\ \uparrow \rule{0mm}{2.5ex} \\  [\Psi] \rule{0mm}{2.5ex}
      \end{matrix} 
   }
  \ar[l]_{ \; P_\Psi \; \leftarrow \; [\Psi] }
}
\]

\section{Duality}
\label{sec:dual}

Recall the sets $\text{\rm AON}_d (\F)$, $\text{\rm SIS}_d (\F)$
from Definition \ref{def:AONSIPS}
and the set $\text{\rm CS}_d (\F)$ from Definition \ref{def:CS}.
So far, we have shown that these three sets are mutually in bijection,
and we described the bijections involved.
For an idempotent system $\Phi$ we have its dual $\Phi^*$ from below Definition \ref{def:ips}.
By construction we have the bijection
$\text{\rm SIS}_d (\F) \to \text{\rm SIS}_d (\F)$,
$[\Phi] \mapsto [\Phi^*]$,
which we call the {\em duality map} on $\text{\rm SIS}_d (\F)$.
The sets $\text{\rm AON}_d (\F)$, $\text{\rm CS}_d (\F)$
each inherit a duality map via the above bijections.
In this section we describe these duality maps in detail.

\begin{defi}    \label{def:dualityAON}    \samepage
\ifDRAFT {\rm def:dualityAON}. \fi
By the {\em duality map on $\text{\rm AON}_d (\F)$}
we mean the composition
\[
 \xymatrix@C=50pt{
  \text{\rm AON}_d (\F)
  \ar[r]_{ P \; \mapsto \; [\Phi_P] }
&
  \text{\rm SIS}_d (\F)
  \ar[r]_{ [\Phi] \; \mapsto \; [\Phi^*] }
&
  \text{\rm SIS}_d (\F)
  \ar[r]_{ [ \Phi ] \; \mapsto \; P_\Phi }
&
 \text{\rm AON}_d (\F).
}
\]
\end{defi}

\begin{theorem}    \label{thm:dualityAON}    \samepage
\ifDRAFT {\rm thm:dualityAON}. \fi
The duality map $\text{\rm AON}_d (\F) \to \text{\rm AON}_d (\F)$
sends $P \mapsto \nu P^{-1}$,
where $\nu^{-1}$ is the $(0,0)$-entry of $P^{-1}$.
\end{theorem}

\begin{proof}
Pick an AO normalized solid invertible matrix $P \in \Matd$,
and set $\Phi = \Phi_P$.
By Definition \ref{def:dualityAON} the duality map
$\text{\rm AON}_d (\F) \to \text{\rm AON}_d (\F)$
sends $P \mapsto P_{\Phi^*}$.
So it suffices to show that $\Phi^*$ has first eigenmatrix $\nu P^{-1}$.
By Proposition \ref{prop:inverse}, $\Phi$ has first eigenmatrix $P$.
Note by Lemma \ref{lem:P0i}(iii) that the scalar $\nu$ is equal to
the size of $\Phi$ from Definition \ref{def:ISnu}.
By Definition \ref{def:P} and Lemma \ref{lem:PQ}(i),
the first eigenmatrix of $\Phi^*$ is $\nu P^{-1}$.
The result follows.
\end{proof}

\begin{defi}    \label{def:dualP}    \samepage
\ifDRAFT {\rm def:dualP}. \fi
Let $P$ denote an AO normalized solid invertible matrix in $\Matd$.
By the {\em dual of $P$} we mean the matrix $\nu P^{-1}$ from
Theorem \ref{thm:dualityAON}.
Let $P^*$ denote the dual of $P$.
\end{defi}

\begin{defi}     \label{def:dualityCS}    \samepage
\ifDRAFT {\rm def:dualityCS}. \fi
By the {\em duality map on $\text{\rm CS}_d (\F)$}
we mean the composition
\[
 \xymatrix@C=50pt{
  \text{\rm CS}_d (\F)
  \ar[r]_{ [\Psi] \; \mapsto \; [\Phi_\Psi] }
&
  \text{\rm SIS}_d (\F)
  \ar[r]_{ [\Phi] \; \mapsto \; [\Phi^*] }
&
  \text{\rm SIS}_d (\F)
  \ar[r]_{ [ \Phi ] \; \mapsto \; [\Psi_\Phi] }
&
 \text{\rm CS}_d (\F).
}
\]
\end{defi}

\begin{defi}    {\rm \cite[Definition 3.2]{Egge} }
\label{def:dualPsi}    \samepage
\ifDRAFT {\rm def:dualPsi}. \fi
Let $\Psi$ (resp.\ $\Psi'$) denote a character system over $\F$ 
with eigenmatrix $P$ (resp.\ $P'$).
We say that $\Psi$ and $\Psi'$ are {\em dual} whenever $P P' \in \F I$.
\end{defi}

\begin{lemma}    \label{lem:dualPsipre}     \samepage
\ifDRAFT {\rm lem:dualPsipre}. \fi
Referring to Definition \ref{def:dualPsi},
assume that $\Psi$, $\Psi'$ are dual.
Then $P' = P^*$.
\end{lemma}

\begin{proof}
Use Lemmas \ref{lem:xie02}(i) and \ref{lem:PAOpre}(i).
\end{proof}

The following lemma is a variation on  a result by Kawada, see \cite[Theorem II.5.9]{BI}.

\begin{lemma}    \label{lem:dualPsi}    \samepage
\ifDRAFT {\rm lem:dualPsi}. \fi
For each character system over $\F$,
its dual exists and is unique up to isomorphism of character systems.
\end{lemma}

\begin{proof}
Concerning existence, let $\Psi$ denote a character system over $\F$ 
with eigenmatrix $P$.
By Theorem \ref{thm:dualityAON},
$P^*$ is AO normalized solid invertible,
so by Lemma \ref{lem:PsiP}, $\Psi_{P^*}$ is a character system over $\F$
that has eigenmatrix $P^*$.
By construction $P P^* \in \F I$ so $\Psi$, $\Psi_{P^*}$ are dual.
We have shown existence.
The uniqueness assertion follows from Proposition \ref{prop:CalgP} and
Lemma \ref{lem:dualPsipre}.
\end{proof}

\begin{defi}    \label{def:dualclass}     \samepage
\ifDRAFT {\rm def:dualclass}. \fi
A pair of isomorphism classes in $\text{\rm CS}_d (\F)$ are said to be
{\em dual} whenever each character system in the first isomorphism class
is dual to each character system in the second isomorphism class.
\end{defi}

\begin{lemma}    \label{lem:dualPsi2}    \samepage
\ifDRAFT {\rm lem:dualPsi2}. \fi
For each isomorphism class in $\text{\rm CS}_d (\F)$ there exists a
unique dual isomorphism class in $\text{\rm CS}_d (\F)$.
\end{lemma}

\begin{proof}
By Lemma \ref{lem:dualPsi}.
\end{proof}

\begin{theorem}    \label{thm:dualityCS}    \samepage
\ifDRAFT {\rm thm:dualityCS}. \fi
The duality map $\text{\rm CS}_d (\F) \to \text{\rm CS}_d (\F)$
sends each isomorphism class in $\text{\rm CS}_d (\F)$ to the dual isomorphism
class in $\text{\rm CS}_d (\F)$.
\end{theorem}

\begin{proof}
By construction the composition
\begin{equation}              \label{eq:comp}
\xymatrix@C=45pt{
 \text{\rm CS}_d (\F)
 \ar[r]_{ [\Psi]  \; \mapsto \; P_\Psi }
&
 \text{\rm AON}_d (\F)
 \ar[r]_{ P \; \mapsto \; P^* }
&
 \text{\rm AON}_d (\F)
 \ar[r]_{ P \; \mapsto \; [ \Psi_P ] }
&
 \text{\rm CS}_d (\F)
}
\end{equation}
sends each isomorphism class in $\text{\rm CS}_d (\F)$
to the dual isomorphism class in $\text{\rm CS}_d (\F)$.
By Proposition \ref{prop:main4}(i) and Definition \ref{def:dualityCS}
the map \eqref{eq:comp} is equal to the duality map on $\text{\rm CS}_d (\F)$.
The result follows.
\end{proof}

\section{Acknowledgement}

The authors thank Harvey Blau for giving the paper a close reading and offering many
valuable suggestions.

\bigskip\bigskip\noindent
Kazumasa Nomura\\
Tokyo Medical and Dental University\\
Kohnodai, Ichikawa, 272-0827 Japan\\
email: knomura@pop11.odn.ne.jp

\bigskip\noindent
Paul Terwilliger\\
Department of Mathematics\\
University of Wisconsin\\
480 Lincoln Drive\\ 
Madison, Wisconsin, 53706 USA\\
email: terwilli@math.wisc.edu

\bigskip\noindent
{\bf Keywords.}
Association scheme,  
Bose-Mesner algebra, character algebra, idempotent system. 
\\
\noindent
{\bf 2020 Mathematics Subject Classification.} 
05E30,  15A21, 15B10.

\end{document}